\documentclass[a4paper,twoside,reqno,11pt,american]{amsart} 

\usepackage[top=2.5cm, bottom=2.5cm, left=2.5cm, right=2.5cm]{geometry}
\usepackage[utf8]{inputenc}
\usepackage[T1]{fontenc}

\usepackage[provide=*]{babel} 
\usepackage{csquotes} 
\usepackage{comment} 
\usepackage{wrapfig}

\usepackage[noindentafter]{titlesec}
\titleformat{name=\section}[hang]{\scshape \large \centering}{\thesection.~}{0em}{}[]
\titleformat{\subsection}[hang]{\bfseries}{\thesubsection.~}{0em}{}[]
\titleformat{\subsubsection}[runin]{\bfseries}{\thesubsubsection.}{.2em}{}[]

\usepackage{titletoc}
\titlecontents{section}[2pc]{\addvspace{1pc}}
{\contentslabel[\textsc{\thecontentslabel}]{1pc}}
{}{\dotfill\contentspage}
[\addvspace{2pt}]

\titlecontents{subsection}[4pc]{}
{\contentslabel[\subsectionname\thecontentslabel]{2pc}}
{}{\hfill\contentspage}
[]

\usepackage{mathtools,amssymb,amsthm,latexsym}
\usepackage{stmaryrd} 
\newcommand{\myfatslash}{\mathbin{\mkern-6mu\fatslash}}
\usepackage{cases} 

\usepackage[arrow,matrix]{xy} 
\usepackage{tikz-cd}
\usepackage[shortlabels]{enumitem} 

\usepackage{graphicx}
\usepackage[dvipsnames]{xcolor}
\usepackage[dvipsnames]{xcolor}
\definecolor{e-mail}{rgb}{0,.40,.80}
\definecolor{reference}{rgb}{.20,.60,.22}
\definecolor{mrnumber}{rgb}{.80,.40,0}
\definecolor{citation}{rgb}{0,.40,.80}

\usepackage{enumitem}
\setlist[enumerate]{label*=(\roman*)} 

\usepackage[
backend=biber, 
style=alphabetic,
sorting=nyvt,
firstinits=true,
doi=false,
isbn=false,
url=false,
backref=true
]{biblatex}
\usepackage[
bookmarksnumbered=true,
colorlinks=true,
linkcolor=reference, citecolor=citation, urlcolor=e-mail
]{hyperref}

\usepackage[noabbrev, capitalise, nameinlink]{cleveref} 
\crefname{equation}{}{}
\crefname{item}{}{}
\crefname{enumi}{}{}
\crefname{enumii}{}{}
\crefname{enumiii}{}{}
\crefname{enumiv}{}{}

\theoremstyle{plain}
\newtheorem{theorem}{Theorem}[section]
\newtheorem{proposition}[theorem]{Proposition}
\newtheorem{lemma}[theorem]{Lemma}
\newtheorem{corollary}[theorem]{Corollary}
\newtheorem*{theorem*}{Theorem}

\theoremstyle{definition}
\newtheorem{definition}[theorem]{Definition}
\newtheorem{example}[theorem]{Example}

\theoremstyle{remark}
\newtheorem{remark}[theorem]{Remark}

\numberwithin{equation}{section}

\newcommand{\XX}{\mathcal{X}}
\newcommand{\XXrig}{\XX_{\mathrm{rig}}}
\newcommand{\XXcoa}{\XX_{\mathrm{coarse}}}

\newcommand{\YY}{\mathcal{Y}}
\newcommand{\YYcoa}{\YY_{\mathrm{coarse}}}
\newcommand{\YYrig}{\YY_{\mathrm{rig}}}
\newcommand{\II}{\mathcal{I}}
\newcommand{\MM}{\mathcal{M}}

\newcommand{\ZZ}{\mathbf{Z}}
\newcommand{\CC}{\mathbf{C}}
\newcommand{\FF}{\mathbf{F}}
\newcommand{\QQ}{\mathbf{Q}}

\newcommand{\spec}{\operatorname{Spec}}
\newcommand{\Hom}{\operatorname{Hom}}
\newcommand{\Isom}{\operatorname{Isom}}
\newcommand{\Aut}{\operatorname{Aut}}
\newcommand{\Out}{\operatorname{Out}}
\newcommand{\Card}{\operatorname{Card}}

\newcommand{\piet}{\pi_1^{\rm et}}
\newcommand{\Gk}{\mathrm{G}_K}

\newcommand{\Dab}{\Delta^{\rm ab}}
\newcommand{\DOne}{\Delta^{(1)}}
\newcommand{\Dtf}{\Delta_{\rm tf}}          

\newcommand{\ab}{\mathrm{ab}}
\newcommand{\DerivSeries}[2]{#1^{(#2)}}
\newcommand{\FinStepSolvQuo}[2]{#1^{#2}}
\newcommand{\CenterSubgrp}[1]{\mathrm{Z}\left(#1\right)}

\newcommand{\R}{\widetilde{\rm Out}_{\pi}(\II_{\XX,\eta})} 
\newcommand{\T}{\widetilde{\Out}^{\Gk}_\pi(\II_\eta,\XXrig)}

\title{Anabelian geometry for Deligne-Mumford curves}

\author[B.~Collas]{Benjamin~Collas}
\address{Research Institute for Mathematical Sciences, Kyoto University, Kyoto 606-8502, Japan}
\email{bcollas@kurims.kyoto-u.ac.jp}

\author[S.~Philip]{S\'everin~Philip}
\address{Stockholm University, Department of Mathematics, Albanovägen 28, SE-106 91 Stockholm, Sweden}
\email{severin.philip@math.su.se}

\author[N.~Yamaguchi]{Naganori~Yamaguchi}
\address{Institute of Science Tokyo, 2-12-1 Ookayama, Meguro-ku, Tokyo 152-8550, Japan}
\email{yamaguchi.n.ac@m.titech.ac.jp}

\subjclass[2020]{Primary: 14A20, 14H30, 20E18; Secondary: 14G32, 20F34, 14H37 }

\keywords{Anabelian geometry; Deligne–Mumford curves; stacky curves and orbicurves, profinite groups; étale fundamental group; Fenchel–Nielsen theory; solvable quotients; hyperbolic curves; stack inertia; Grothendieck conjecture}

\thanks{This manuscript is part of the France-Japan AHGT international research network supported by the Research Institute for Mathematical Sciences of Kyoto University and CNRS. The first author acknowledges the support of the International Center for Next-Generation Geometry, a center affiliated with the Research Institute for Mathematical Sciences located in Kyoto University. The second author is supported by the Verg Foundation and the Royal Swedish Academy of Sciences. The third author would like to express sincere gratitude to Prof. Akio Tamagawa for his invaluable assistance and insightful suggestions throughout this research.}

\begin{document}
	
	\begin{abstract}
		We develop an anabelian framework for general Deligne–Mumford curves, showing that their stack and orbifold  structures are encoded in the group-theoretic properties of their étale fundamental groups. After establishing the required properties for profinite $F$-groups, we prove that fundamental geometric features, including hyperbolicity, affineness, and inertia data, can already be detected from low-level solvable quotients of the associated profinite groups, namely at the optimal $3$-step level. As a consequence, we obtain some anabelian reconstruction results for Deligne–Mumford curves, their rigidifications, and their coarsification. While the $m$-step Grothendieck conjecture doesn't hold for Deligne-Mumford curves, we establish a $5$-step anabelian theorem for the rigidification of affine Deligne-Mumford curves, namely affine stacky curves. A certain emphasis is given to the role of stack inertia groups.
	\end{abstract}
	
	\maketitle
	
	\vfill
	
	\tableofcontents
	\thispagestyle{empty}
	
	\vspace{2em}
	
	\newpage

		\section{Introduction}\label{intro}
		This paper initiates a general anabelian framework for Deligne-Mumford curves over fields of characteristic zero. We establish some reconstruction algorithms for certain arithmetic-geometric data following two \emph{a priori} distinct but indeed related guidelines, which are, the use of ``higher'' stack symmetries for curves (or stack inertia groups), and the use of ``lower'' derived quotients of homotopy groups (or $m$-step solvable properties). 
		
		This approach results in a shift of paradigm that takes the hidden symmetries of curves into account. As a consequence:\textbf{(a)}~anabelian properties hold for curves beyond the classical scheme hyperbolic condition, \textbf{(b)}~cuspidal and stack inertia conditions are mixed together, and \textbf{(c)}~the Grothendieck $m$-step solvable conjecture doesn't hold as it is. Constructions of this paper thus provide new arithmetic-geometric insights in the case of the moduli spaces of curves -- whose anabelian character is expected but, beyond the genus zero case corresponding to configuration spaces, is not established; we refer to \cite{Col26} Section~2.1.2, which puts together \cite{Nak94} and \cite{MR1478817}, and more generally to \cite{HMM22} Theorem~A for the configuration spaces of general hyperbolic curves -- and in the case of the weighted projective line.
		
		We further develop and establish some purely group-theoretic property of Fenchel groups -- or $F$-groups, which correspond to the geometric fundamental group of the rigidification of a Deligne-Mumford curve -- which are of independent interest, and we enrich some previous purely geometric constructions of stacky curves -- such as given in \cite{MR2279100} and \cite{GS17}, and as gathered in \cite{VZB22} -- with their arithmetic and homotopical nature.
		
		\subsection{Stack anabelian arithmetic geometry}
		In its modern form and its general mono-version, anabelian geometry focuses on the canonical reconstruction of discrete invariants of a space $X$ -- such as the topological type $(g,r)$ of a hyperbolic curve, or its cuspidal inertia/decomposition groups corresponding to closed points -- from its étale (resp. geometric) fundamental group $\Pi_X$ (resp. $\Delta_X$); it includes the reconstruction of $X$ itself. The original \emph{Grothendieck conjecture} states, for given $X$ and $Y$ spaces over a field $K$, the bijectivity of the natural map
		\[
		\Isom_K(Y,X)\longrightarrow \Isom_{G_K}(\Pi_{Y},\Pi_{X})/\sim_{\Delta_{X}}.
		\]
		where the right-hand side denotes the $\Delta_X$-orbits of $G_K$-compatible isomorphisms $\Pi_Y\to\Pi_X$. In its weak form, the Grothendieck conjecture deals with the existence of a map in the reverse direction.
		
		In the case of hyperbolic curves, this conjecture was successively established by H.~Nakamura in genus zero (weak form) \cite{MR1072981} Theorem~6.1, by A.~Tamagawa for affine curves \cite{MR1478817} Theorem~6.3, and by S.~Mochizuki in its general form \cite{MR1720187} Theorem~16.5.
		
		\medskip
		
		\subsubsection{}
		Let $\XX$ be a Deligne-Mumford curve, which comes with its rigidification $\XXrig$ as defined in \cite{AOV08} -- also called a stacky curve, that has no non-trivial generic stack inertia -- and its coarsification $\XXcoa$ -- which is a (smooth) scheme amounting, in the case of an underlying moduli problem, to resolving the moduli problem attached to $\XX$ only geometrically. In our context, the weak Grothendieck conjecture for Deligne-Mumford curves can then be seen as structured in a $3$-level diagram as follows:
		
		\begin{equation*}
			\begin{tikzcd}
				\Isom_K(\YY,\XX) \ar[r,shift left,"\Psi"]\ar[d]& \Isom_{G_K}(\Pi_{\YY},\Pi_{\XX})/\sim_{\Delta_{\XX}}\ar[d,two heads]\ar[l,shift left,dashed]\\
				\Isom_K(\YYrig,\XXrig) \ar[r,"\Psi_{\rm rig}", shift left]\ar[d, shift right]& \Isom_{G_K}(\Pi_{\YYrig},\Pi_{\XXrig})/\sim_{\Delta_{\XXrig}} \ar[d,two heads]\ar[l,shift left, dashed]\\
				\Isom_K(\YYcoa,\XXcoa) \ar[r,"\Psi_{\rm coarse}",shift left,]\ar[u,shift right, dashed]&\Isom_{G_K}(\Pi_{\YYcoa},\Pi_{\XXcoa})/\sim_{\Delta_{\XXcoa}} \ar[l,shift left, dashed]
			\end{tikzcd}
		\end{equation*}
		where dashed arrows indicate progress established in this manuscript.
		\medskip
		
		Our first main result is given as follows -- see Proposition~\ref{prop:IrigcoaRec} and Theorem~\ref{theo:DMbiAnab}.
		\begin{theorem}\label{Tha:MainA}
			For $\XX$ a hyperbolic Deligne-Mumford curve over a number field, the weak Grothendieck conjecture 
			for $\XX$ and for its rigidification $\XXrig$ holds.
		\end{theorem}
		
		This series of results relies, at the rigidified level, on the profinite $F$-group theory developed in Section~\ref{subsub:SlimCfree}, and follows Hoshi's Grothendieck conjecture (weak form) for stacky curves of \cite{Hos22} Theorem~4.5. While the latter follows from a Galois descent of the classical Grothendieck conjecture for curves, we provide a separate proof based on the mono-anabelian property of the stack inertia groups $\II_{\XX,\bar{x}}$ at closed points -- see Proposition~\ref{prop:IrigcoaRec} and \ref{prop:BianabXrig}. Following Mochizuki's philosophy of arithmetic holomorphy, this latter approach allows, over $p$-adic fields, $\XXrig$ varying in a mono-anabelian holomorphic deformation around the fixed discrete $\II_{\XX,\bar{x}}$ data -- see \cite{Col26} \S.~1.3.1 and the original \cite{AbsTopIII} Introduction for this principle -- which can be the object of future work.
		
		The anabelian result for Deligne-Mumford curves follows the step-by-step identification of $\XX$ as the global quotient of a scheme by a finite group and as a $\II_{\eta}$-gerbe over $\XXrig$ \emph{à la} Giraud, see Proposition~\ref{prop:monoXoverXrig}, which is an algebraic and homotopical refinement of Behrend-Noohi in \cite{MR2279100} Proposition~7.5.
		
		\medskip
		
		Due to the stack inertia, the hyperbolic condition for Deligne-Mumford curves is weaker than the hyperbolic condition for curves -- that is $2-2g-r-\sum(1-1/n_i)<0$, where $n_i$ denotes the order of stacky points. Our stack approach thus establishes the weak Grothendieck conjecture for (coarse) smooth curves whose topological data $(g,r)$ are beyond the usual scheme hyperbolic condition $2-2g-r<0$ -- see Remark~\ref{rem:hypRigAnabRel}~(i).

		\subsubsection{} The motivation for the use of derived series $\{\Delta_{\XXrig}^{(m)}\}_{m\geq 1}$ and $m$-step maximal quotients $\Delta_{\XXrig}^m=\Delta_{\XXrig}/ \Delta_{\XXrig}^{(m)}$ for our stacky anabelian study comes from the homotopic nature of the stack inertia groups $\II_{\XXrig,\bar{x}}$, that is $\II_{\XXrig,\bar{x}}<\Pi_{\XXrig}$ for $\bar{x}$ a geometric point of $\XXrig$ (see Section~\ref{subsub:InertiaHomotopy} and also Section~\ref{sub:introIC} below), and a certain abelian-period condition as in Section~\ref{sub:torsfreecommutator} that detects how the stack inertia groups survive in the (maximal abelian) $1$-step quotient $\Delta_{\XXrig}^1$.

		\medskip
		
		In the anabelian reconstruction process for stacky curves -- and in contrast to the case of general Deligne-Mumford curves -- the stack inertia intervenes at the first layer out of five; that is: ($m=1$)~deals with the stack inertia property (and the genus), $(m=3)$~deals with geometric properties (hyperbolicity and affineness), and $(m\geq 5)$~finalizes the reconstruction. More precisely, denoting $\Pi^{\Delta-m}_{\XX}=\Pi_{\XX}/\Delta_{\XX}^{(m)}$ in accordance with the action of $G_K$ on $\Delta_{\XX}<\Pi_{\XX}$, our second main series of results can be stated as follows -- see Theorem~\ref{thm:mstep_generalGenus} with Remark~\ref{rem:asbmStep}, Section~\ref{subsub:anabStepGrp} for affineness and hyperbolicity, and more generally Section~\ref{subsection:2.2} 
		\begin{theorem}\label{Th:MainB} For a \emph{ hyperbolic Deligne-Mumford curve}, the Grothendieck $m$-step conjecture doesn't hold for any $m\geq 1$. On the other hand, for $\XX$ and $\XX'$ \emph{stacky curves} over $K$ a finitely generated field over $\QQ$, which are affine, non-perfect, and hyperbolic, with $(g_\XX,r_\XX)\neq (0,1)$, the Grothendieck $5$-step conjecture holds, that is:
			\begin{equation}\tag*{$(*)_{m=5}$}
				\mathcal{X}\cong\mathcal{X}' \text{ if and only if }\Pi_{\mathcal{X}}^{\Delta-5}\cong\Pi_{\mathcal{X}'}^{\Delta-5}
			\end{equation}
			The level $m=3$ is furthermore optimal for characterizing the hyperbolicity and the affineness of~$\XX$.
		\end{theorem}
		
		The $m$-step part of this result relies on a $4$-step weak Grothendieck conjecture result for \emph{affine} curves of \cite{MR4745885} by specialization to generic fiber that echoes Oda-Tamagawa's good reduction criterion of \cite{MR1478817}, and on a certain cover provided by Theorem~\ref{theo:mainC}; passing from $\XXrig$ to a finite étale schematic cover as given by Theorem~\ref{theo:mainC} above adds another step. The case of proper curves would require adapting distinct techniques from $p$-adic Hodge theory and line bundles as in \cite{MR1720187}.
		
		In the case of Deligne-Mumford curves, on the other hand, the non-solvability of the generic stack inertia stands as a general obstruction to the $m$-step anabelianity of Deligne-Mumford curves -- see Section~\ref{subsub:mGCfails}. 
		
		\medskip
		
		Note that since the Grothendieck $(m-1)$-step conjecture implies the Grothendieck $m$-step conjecture, and all imply the weak form of the classical Grothendieck conjecture, the result above in particular recovers \cite{Hos22} Theorem~4.5 and the rigidification part of Theorem~\ref{Tha:MainA}.
		
		\medskip
		
		Let us recall that this result is part of a broader series of progress: for number fields ($m$-step solvable for Neukirch-Uchida's Theorem for $m=3$, Saïdi-Tamagawa \cite{ST22}), for affine hyperbolic curves ($m=4$, see \cite{MR4745885}), and for proper hyperbolic curves ($m=5$, Mochizuki \cite{MR1720187} Theorem~18.1) -- for a similar birational almost-abelian approach, we also refer to \cite{Top16}.

		\subsection{Fenchel groups: profinite and truncated properties}
		In the profinite group-theoretic settings, establishing the $m$-step anabelian properties of stacky curves results in the development of further notions and properties for $F$-groups.
		
		\medskip
		
		\subsubsection{} The fact that a hyperbolic profinite $F$-group $\Delta$ is ``globally'' not solvable (it contains a torsion-free $F$-group that is not, see Theorem~\ref{theo:mainC}) but ``locally'' solvable (it contains solvable quotients of any derived length, see (ii)~below) leads to the intermediate notions of $m$-solvability and $m$-derived perfectness. 
		
		
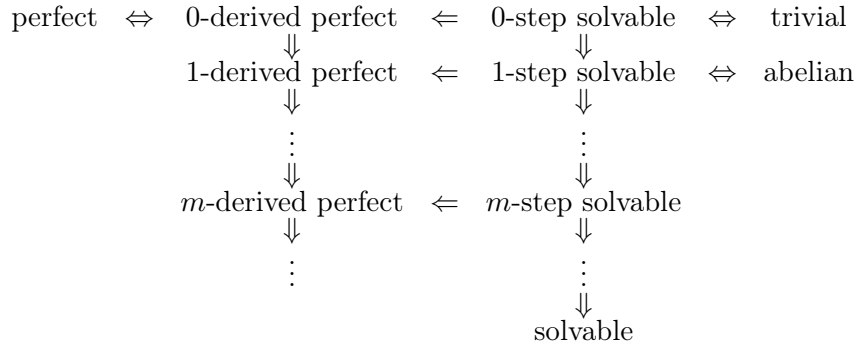
\begin{center}
			\begin{equation*}\def\arraystretch{.7}
				\begin{array}{ccccccc}
					\text{perfect} & \Leftrightarrow & 0\text{-derived perfect}
					& \Leftarrow      & 0\text{-step solvable}  & \Leftrightarrow & \text{trivial}                                            \\
					&                 & \Downarrow              &                 & \Downarrow             &                 &                \\
					&                 & 1\text{-derived perfect} & \Leftarrow      & 1\text{-step solvable} & \Leftrightarrow & \text{abelian} \\
					&                 & \Downarrow              &                 & \Downarrow             &                 &                \\
					&                 & \vdots                  &                 & \vdots                 &                 &                \\
					&                 & \Downarrow              &                 & \Downarrow             &                 &                \\
					&                 & m\text{-derived perfect} & \Leftarrow      & m\text{-step solvable} &                 &                \\
					&                 & \Downarrow              &                 & \Downarrow             &                 &                \\
					&                 & \vdots                  &                 & \vdots                 &                 &                \\
					&                 &                         &                 & \Downarrow             &                 &                \\
					&                 &                         &                 & \text{solvable}        &                 &
				\end{array}
			\end{equation*}
		\end{center}
		
		\medskip
		
		The $m$-derived perfectness stands as an obstruction to checking if the stack inertia survives in the maximal $(m+1)$-step solvable quotient $\Delta^{m+1}$ -- see Section~\ref{subsub:mStpSol} and also the structure result of Lemma~\ref{lem:weaklyprefect_transitive}.
		
		\subsubsection{}
		The following $m$-step profinite Fenchel-Nielsen result, see Theorem~\ref{thm:proFFenchels}, in its group-theoretic proof and formulation, is a key ingredient in establishing the anabelian $m$-step Theorem~\ref{Th:MainB} and the group-theoretic properties (i)-(iii) below.
		
		\begin{theorem}\label{theo:mainC}
			Every non-perfect profinite $F$-group $\Delta$ admits a torsion-free open characteristic subgroup whose quotient has derived length at most $3$. When $\Delta$ is affine, or is proper and satisfies a certain abelian-period condition, the quotient can furthermore be supposed abelian.
		\end{theorem}
		This result generalizes and refines with purely group-theoretic argument a previous discrete version -- established in terms of Fuchsian and triangle groups, see \cite{MR2379717} Theorem~3.1 and \cite{MR4816529} Theorem~1.1, and more generally \cite{Sah69}. It follows from a profinite version of the classcial Fenchel-Nielsen Theorem that is given in Theorem~\ref{thm:classicalproFfenchel}.
		
		\medskip
		
		Note that in the proper case -- that is with no cuspidal data -- the abelian-period condition amounts to preserving the cyclic stack inertia groups in the derived group, see Section~\ref{subsub:abelPer}.
		
		\subsubsection{}
		In addition to the previously mentioned characterization of the hyperbolicity by the maximal $3$-step solvable quotient $\Delta^3$ of $\Delta$, we further obtain the following group-theoretic properties:
		\begin{enumerate}
			\item The list of periods, the sign of the Euler characteristic, and the affineness property are isomorphism invariants of the group -- see Proposition~\ref{prop:SigIsoInv};
			\item 	Every non-perfect hyperbolic profinite $F$-group admits a solvable quotient of derived length $c$ for any integers $c\geq 0$ -- see Proposition~\ref{cor:infinitederived};
			\item The maximal $m$-step solvable quotients $\Delta^m$ are center-free for $m$ large enough -- see Proposition~\ref{prop:center-free}.
		\end{enumerate}
		These group-theoretic statements, which are systematically exploited for establishing our anabelian results above, can also be seen as an application of anabelian geometry to group theory; they can be of independent interest for the group theorist.

		\subsection{Stack inertia vs cuspidal arithmetic homotopy}\label{sub:introIC}
		As part of a broader project on the stack arithmetic homotopy of moduli spaces of curves $\MM_{g,[r]}$ -- in Grothendieck-Teichmüller theory, where one gives a stack inertia definition of the Grothendieck-Teichmüller group $GR\simeq\rm \!I\Gamma$ in \cite{Col12}; in Galois-Teichmüller theory, where one establishes that the $G_\QQ$-action on cyclic stack inertia has the same type as on the cuspidal one \cite{CM23}; and with respect to Oda's problem, where stack methods establish a result originally obtained in a cuspidal framework \cite{CP25}  -- this paper realizes a more formal link between stack inertia groups and cuspidal groups. 
		
		\subsubsection{} Conceptually, \emph{the virtual wounding of the curve by the stack inertia} is revealed in the derived group as follows: at the level of the torsion-free derived subgroup $\Delta_\XX^{(1)}$, the stack inertia disappears and resurfaces as cuspidal points, which can then be tracked down at higher $m$-step levels. More precisely, the number $r_{\Delta^{(1)}}$ of cuspidal points is given in terms of the orbifold periods of the stacky points -- see Lemma~\ref{cor:r_greater}. 
		
		\medskip
		
		The stack inertia vs cuspidal inertia duality is more apparent in the following proper vs affine cases. In the case of \emph{proper} profinite $F$-groups, namely of type $(g,r)=(g,0)$, the abelian-period condition is a numerical condition for the cyclic stack inertia to survive at the first stage of the derived series, while in the \emph{affine} case of type $(g,r)=(0,1)$, the number of cusps $r_H$ -- where $H$ corresponds to a schematic étale cover of a stacky curve $\XXrig$ -- is bounded by $r_{\Delta^{(1)}}$, and thus stands as a first $m$-step obstruction to anabelian phenomenon as in Section~\ref{subsection:2.2}.

		\subsubsection{}
		All the anabelian results apply \emph{stricto sensu} to the dimension one moduli spaces of curves --  that is, $\MM_{0,[4]}$ and $\MM_{1,1}$ -- to the dimension one irreducible components of special loci $\MM_{g,[r]}(A)$ of curves of genus $g$ with $r$ marked points admitting $A\leq\II_{\MM,\bar{z}}$ as a stack inertia group, and to their rigidification. 
		
		\emph{Sensu lato}, the higher symmetries of stacky curves are related to the higher symmetries of the moduli space of curves $\MM_{g,[r]}$ as follows. Denote by $(g,r)$ the topological type of $\XXrig$. Since $\XXrig$ contains a dense schematic curve $X$, a surjective morphism $\Pi_{\XXrig}\twoheadrightarrow A$ onto a finite abelian group (such as involved in the abelian-period condition with respect to $\Delta_{\XXrig}^1$) defines a certain $A$-cover $Z\to X$, that is, a point $z\colon \spec K\to \MM_{g,[r]}(A)$ of the moduli spaces of curves of genus $g$ with $r$ marked points, whose stack inertia groups satisfy $A\leq\II_{\MM,\bar{z}}<\Pi_{\MM}$. We further refer to \cite{CM23} and Section~4.2 herein for the arithmetic of this stack inertia in the case where $A$ is a cyclic group.

		\subsubsection{} 
		As investigating remarks, let us note that this cuspidal-stack inertia behavior would suggest that the stack inertia should naturally appear in the case of the $m$-step Grothendieck conjecture for proper hyperbolic non-perfect stacky curves. In the affine case, the use of stack inertia-only arguments, as a replacement to the use of Theorem~\ref{theo:mainC}, could also lower the bound from $m=5$ to $m=4$.
		
		\medskip
		
		In another direction, note that a straightforward application of Theorem~\ref{Th:MainB} to hyperbolic smooth weighted projective lines $\mathbf{P}(a_1,\dots,a_n)$ with $a_i\in\mathbf{Z}_{>0}$, provides the examples of \emph{Fano varieties with log terminal singularities} (cf.~\cite{Dol82,KM98}) that exhibit some anabelian behavior. With this respect, the results of this paper thus further point toward investigating the group-theoretic and homotopical nature of classical structures arising in representation theory (canonical algebras and Auslander–Reiten theory~\cite{GL87}), singularity theory (orbifold weights and graded rings), and derived geometry (derived categories and Bridgeland stability conditions~\cite{Bri07}).

		
	\bigskip

	{\footnotesize
		\noindent \emph{Notations}. 
		For $G$ a profinite group, we follow (unlike the classical anabelian papers) the usual group theoretic notations $\{G^{(k)}\}_{k\geq 0}$ for derived series, where  $G^{(0)} \coloneqq G$, and $G^{(m)}$ is the closure of $[G^{(m-1)},G^{(m-1)}]$ for  $m\geq 1$. The \emph{derived length of $G$}, if it exists, is the minimal integer $k$ such that $G^{(k)}=1$; the \emph{maximal $m$-step solvable quotient of $G$} is defined as $G^{m} \coloneqq G / G^{(m)}$. A group $G$ is said to be \emph{perfect} if $G{\color{red}\simeq}{\color{blue}=} G^{(1)}$. We denote by $G_{\rm tf}$ the maximal torsion-free quotient of $G$.
	}
	
	\section{Group theoretic properties of profinite \texorpdfstring{$F$}{F}-groups}\label{sec:Fenchel}
	
	The notions of profinite $F$-group (where $F$ stands for Fenchel), and their related open subgroups, are the group-theoretic notions corresponding to the geometric fundamental group $\Delta_{\XXrig}$ of the rigidification of a Deligne-Mumford curve $\XX$ as in Section~\ref{sub:AHGDM} -- that is, of a smooth $1$-dimensional stack of finite type with trivial generic inertia -- and its finite étale covers. After a brief reminder on the basics of those groups, such as presentation, signature, Euler characteristic, hyperbolicity, and affineness -- which we show by anabelian methods that some of those are group isomorphism invariants in Proposition~\ref{prop:SigIsoInv} --  we establish our Fenchel-Nielsen results in their profinite and derived forms -- see Theorem~\ref{thm:classicalproFfenchel} and Theorem~\ref{thm:proFFenchels} respectively.
	
	\medskip
	
	We also provide a group-theoretic criterion for the stack inertia to survive in the derived series, establish some anabelian slimness and center-freeness properties, as well as a group-theoretic characterization of hyperbolicity and affineness for non-perfect hyperbolic $F$-groups in their $3$-step quotient -- see Section~\ref{sec:weaklyperfect-hyperbolic}. We also note that the result does not hold at the $2$-step level.
	
	\medskip
	
	This section serves as a preparatory group-theoretic introduction to the anabelian results on Deligne-Mumford curves of Section~\ref{sub:GCDM} and the group-theoretic foundation of the $m$-step results for \emph{stacky} curves of Section~\ref{subsection:2.2}.

	\subsection{Profinite \texorpdfstring{$F$}{F}-groups and hyperbolicity}\label{sub:firstpropFgroups}
	
	\subsubsection{} A profinite $F$-group corresponds to the geometric fundamental group of orbicurves. Since topologically finitely generated, such a group can be obtained as the profinite completion of a discrete Fenchel-group $\tilde{\Delta}$ -- see \cite{MR577064} p.~126 for a definition in terms of Fuchsian groups. We recall that discrete $F$-groups are residually finite, that is $\tilde{\Delta}\hookrightarrow {\Delta}$ -- see \cite{Sah69} Theorem~1.5~(a).
	
	\begin{definition}\label{def:proFgroup}
		A profinite group $\Delta$ is called a profinite Fenchel group, or profinite $F$-group, if it admits a topological presentation of the form
		\begin{equation}\label{equation:01-decomposition}\small
			\Bigl\langle \alpha_{1}, \beta_{1}, \dots, \alpha_{g}, \beta_{g}, \gamma_{1}, \dots, \gamma_{r}, \delta_{1}, \dots, \delta_{k} \Bigm| \prod_{i=1}^{g} [\alpha_i, \beta_i] \cdot \prod_{j=1}^{r} \gamma_j \cdot \prod_{\ell=1}^{k} \delta_\ell = 1,\quad \delta_1^{n_1} = \cdots = \delta_k^{n_k} = 1 \Bigr\rangle
		\end{equation}
		for some integers $g$, $r$, $k\in\mathbb{Z}_{\geq 0}$, and $n_i \in\mathbb{Z}_{\geq 2}$ for each $i=1,\dots,k$.
		The tuple
		\begin{equation*}
			\Sigma_{\Delta}=(g, r; \{n_1, \dots, n_k\})
		\end{equation*}
		is called the signature of $\Delta$; here $g$ is the genus, $r$ the number of cusps, and $k$ the number of orbifold points of periods $n_1, \dots, n_k$.
	\end{definition}
	
	In a consistent manner with the geometric context \cite{Wal61}, the Euler characteristic $\chi(\Delta)$ of a profinite $F$-group $\Delta$ with signature $(g, r; \{n_1, \dots, n_k\})$ is then given by
	\begin{equation}\label{def:euler-characteristic}
		\chi(\Delta)
		\coloneqq 2 - 2g - r - \sum_{i=1}^{k} \left(1 - \frac{1}{n_i}\right).
	\end{equation}
	We say that $\Delta$ is \textit{hyperbolic} if $\chi(\Delta)<0$.
	
	\begin{remark}\label{rem:classical_reconstruction_of_euler}
		The presentation (and hence the signature) of a non-trivial profinite $F$-group is not necessarily unique.
		For example, the free profinite group of rank $4$ can be seen as a profinite $F$-group with signature $(2,1;\emptyset)$ or $(0,5;\emptyset)$.
	\end{remark}
	
	\begin{proposition} \label{prop:inducedsign}
		Let $\Delta$ be a profinite $F$-group with signature $(g,r;\{n_1,\dots, n_k\})$. Let $H$ be an open subgroup of $\Delta$. Then $H$ is a profinite $F$-group with signature $(g_H,r_H; \{n_{H,1},\dots, n_{H,k_H}\})$ whose parameters are defined as follows:
		\begin{enumerate}
			\item The number $r_H$ of cusps is given by $r_H=\sum_{i=1}^r \operatorname{cyc}(\gamma_i)$, where $\operatorname{cyc}(\gamma_i)$ denotes the number of disjoint cycles of the permutation induced by $\gamma_i$ in the quotient $\Delta/H$.
			
			\item The number $k_H$ of orbifold points is given by $k_H=\sum_{i=1}^k \operatorname{cyc}(\delta_i)$, where $\operatorname{cyc}(\delta_i)$ denotes the number of disjoint cycles of the permutation induced by $\delta_i$ in the quotient $\Delta/H$.
			
			\item The orbifold data $\{n_i\}_{i=1,\dots,k}$ identifies as a multiset as $\{n_{H,1},\dots, n_{H,k_H}\}=\{m_{i,j}\}_{i,j}$ where $m_{i,j}= n_i/s_{i,j}$ and $s_{i,j}$ denotes the order of the corresponding cycle in the decomposition of the permutation induced by $\delta_i$.
			
			\item\label{RH_formula} The genus $g_{H}$ is the unique rational number that satisfies the equation $\chi(H)= [\Delta:H]\cdot\chi(\Delta)$ given by Riemann-Hurwitz.
		\end{enumerate}
		The tuple $(g_H,r_H; \{n_{H,1},\dots, n_{H,k_H}\})$ is called the \emph{induced signature $\Sigma_H^\Delta$ of $\Delta$ to $H$} 
	\end{proposition}
	
	\begin{proof} Since $\Delta$ is the profinite completion of a discrete $F$-group $\tilde{\Delta}$, and $F$-groups are residually finite, that is $\tilde{\Delta}\hookrightarrow \Delta$, one can assume that $\Delta$ is discrete. In this case, the induced representation of $H$ can be computed via the Reidemeister--Schreier rewriting process.
	\end{proof}
	
	It follows directly that a profinite $F$-group $\Delta$ has no cusps if and only if all its open subgroups have no cusps.
	
	\subsubsection{} In its discrete original form, Fenchel's conjecture states that every orbicurve admits a topological finite cover by a variety; it was successively established by Bundgaard–Nielsen and Fox \cite{MR48447,MR53937}, building on earlier work of Nielsen \cite{MR29378} -- see also \cite{MR702279} for a last erratum. The profinite version is a direct consequence of the discrete one.
	
	\begin{theorem}[Profinite Fenchel-Nielsen]\label{thm:classicalproFfenchel}
		A profinite $F$-group contains a torsion-free open normal subgroup.
	\end{theorem}
	
	\begin{proof}
		Let $\Delta$ be a profinite $F$-group, which, by definition, we may assume is the profinite completion of a discrete $F$-group $\tilde{\Delta}$. The discrete Fenchel-Nielsen theorem provides a cofinite normal subgroup $H\triangleleft\tilde{\Delta}$ that is torsion-free. Denoting $\overline{H}$ the closure of $H$ in $\Delta$, one has $\tilde{\Delta}/H \xrightarrow{\sim} \Delta/\overline{H}$, thus $\overline{H}$ is an open normal subgroup of $\Delta$. Furthermore, since $\tilde{\Delta}$ is residually finite, the natural map
		$\widehat{H} \to \overline{H} < \Delta$ is an isomorphism
		by \cite{BCR16} Corollary~2.8.
		
		Since $H$ is a finite-index torsion-free subgroup of a discrete $F$-group, the induced presentation of Definition~\ref{def:proFgroup} shows that $H$ is either a finitely generated free group or a compact surface group. The profinite completion of such a group is torsion-free; equivalently, every finite subgroup of $\widehat{H}$ is trivial. Thus $\overline{H}\cong\widehat{H}$ is torsion-free and satisfies the required assumptions.
	\end{proof}
	
	A straightforward application of Theorem~\ref{thm:classicalproFfenchel} provides a group-theoretic characterization of hyperbolicity, parabolicity, or ellipticity of a profinite $F$-group. We recall that a group is said to be virtually abelian if it has a cofinite abelian subgroup.
	
	\begin{corollary}\label{rem:hyperbolic-patterns}
		Let $\Delta$ be a profinite $F$-group of signature $(g, r; \{n_1, \dots, n_k\})$, then $\Delta$ is hyperbolic, that is $\chi(\Delta)<0$, if and only if $\Delta$ is not virtually abelian.
		
		Similarly, $\Delta$ is elliptic, that is $\chi(\Delta)>0$ (resp. parabolic, that is $\chi(\Delta)=0$), if and only if $\Delta$ is finite (resp. $\Delta$ is infinite and virtually abelian).
	\end{corollary}
	
	\begin{proof}
		Let $H$ be the cofinite torsion-free normal subgroup of $\Delta$ given by  Theorem~\ref{thm:classicalproFfenchel}. By Proposition~\ref{prop:inducedsign} \ref{RH_formula}, one has $\chi(\Delta) > 0$ if and only if $\chi(H) > 0$ (resp. $\chi(\Delta) = 0$ if and only if $\chi(H) = 0$), and the corresponding characterization for profinite $F$-groups follows directly from the discrete one.
	\end{proof}
	
	\begin{proposition}\label{prop:SigIsoInv}
		For a non-trivial profinite $F$-group, the list of periods, the Euler characteristic, and the affineness, are isomorphism invariants of the group. More precisely, denoting $\Delta^{\rm ab}_{\rm tf}$  the abelianization of the torsion-free quotient of $\Delta$:
		\begin{equation*}
			\chi(\Delta)=2-\mathrm{rank}_{\mathbb{Z}_{\ell}}(\Delta^{\rm ab}_{\rm tf})-\varepsilon- \sum_{i=1}^{k} \left(1 - \frac{1}{n_i}\right)
		\end{equation*}
		where $\varepsilon=1$ (resp. $\varepsilon=0$) if $\Delta$ is affine (resp. not affine).
	\end{proposition}
	
	The proof relies on some group-theoretic reasoning that is classical in anabelian geometry, see also \S~\ref{subsub:anabStepGrp}. We recall that by Definition~\ref{def:proFgroup} the abelianized $\Dab$ identifies with:
	\begin{equation} \label{eq:DabExpl}
		\Delta^{\mathrm{ab}} \cong
		\begin{cases}
			\bigoplus_{i=1}^g \bigl(\widehat{\mathbb{Z}}\alpha_i \oplus \widehat{\mathbb{Z}}\beta_i \bigr)
			\;\oplus\; \bigoplus_{i=1}^{r-1} \widehat{\mathbb{Z}}\, \gamma_i
			\;\oplus\; \left(\bigoplus_{i=1}^k (\mathbb{Z}/n_i\mathbb{Z}) \cdot \delta_i\right)
			\qquad                       & \text{when $r \geq 1$} \\
			\bigoplus_{i=1}^g \bigl(\widehat{\mathbb{Z}}\alpha_i \oplus \widehat{\mathbb{Z}}\beta_i \bigr)
			\;\oplus\; \left(\prod_{i=1}^{k} \mathbb{Z}/n_i\mathbb{Z}\right) / \langle (1,\dots,1) \rangle.
			\qquad                       & \text{when $r = 0$} 
		\end{cases} 
	\end{equation}

	\begin{proof}
		The non-triviality assumption excludes the first two non-hyperbolic cases of Table~\ref{fig:Table_nonhyperbolic_F_groups}. By \cite{MR2402513} Lemma~2.1~(iv), the closed subgroup generated by some conjugate of  $\delta_{i}$ for some $i$ is characterized as a maximal finite subgroup, and for any distinct indices $i, j$, the intersection $\langle\delta_{i}\rangle\cap \langle\delta_{j}\rangle$ is trivial. Thus, the periods $\{n_1,\dots, n_k\}$ are characterized as the orders of the conjugacy classes of the maximal finite subgroups of $\Delta$.
		
		The affineness follows the existence of a non-abelian free open subgroup $H$ of $\Delta$ given by Theorem~\ref{thm:classicalproFfenchel}: if $r\geq 1$, then by Eq.~\eqref{eq:DabExpl}, $H$ is a non-abelian free profinite group, while otherwise, $H$ is a surface group.
		
		The formula for $\chi$ is then immediate from the presentations of Eq.~\eqref{eq:DabExpl}, with $\varepsilon=1$ if and only if $ \Delta/\Delta_{\mathrm{tor}}$ is a free group as above.
	\end{proof}
	
	\begin{table}[b]
		\begin{tabular}{lccc}\hline
			Signature ($\Sigma$) & Group                                                       & $\chi$          & \text{derived length} \\ \hline
			$(0,0;\emptyset)$    & $\{1\}$                                                     & $2$             & $0$                   \\
			$(0,0;\{n\})$        & $\{1\}$                                                     & $1 + 1/n$       & $0$                   \\
			$(0,0;\{n_1,n_2\})$  & $\mathbb{Z}/\mathrm{gcd}(n_1,n_2)\mathbb{Z}$                & $1/n_1 + 1/n_2$ & $1$                   \\
			$(0,0;\{2,2,n\})$    & $D_n$ (Dihedral group)                                      & $1/n$           & $\leq 2$              \\
			$(0,0;\{2,3,3\})$    & $A_4$                                                       & $1/6$           & $2$                   \\
			$(0,0;\{2,3,4\})$    & $S_4$                                                       & $1/12$          & $3$                   \\
			$(0,0;\{2,3,5\})$    & $A_5$                                                       & $1/30$          & non-solvable          \\
			$(0,1;\emptyset)$    & $\{1\}$                                                     & $1$             & $0$                   \\
			$(0,1;\{n\})$        & $\mathbb{Z}/n\mathbb{Z}$                                    & $1/n$           & $1$                   \\
			$(0,0;\{2,3,6\})$    & $(\mathbb{Z}\times\mathbb{Z})\rtimes\mathbb{Z}/6\mathbb{Z}$ & $0$             & $2$                   \\
			$(0,0;\{2,4,4\})$    & $(\mathbb{Z}\times\mathbb{Z})\rtimes\mathbb{Z}/4\mathbb{Z}$ & $0$             & $2$                   \\
			$(0,0;\{3,3,3\})$    & $(\mathbb{Z}\times\mathbb{Z})\rtimes\mathbb{Z}/3\mathbb{Z}$ & $0$             & $2$                   \\
			$(0,0;\{2,2,2,2\})$  & $\langle a,b,c\mid a^2,b^2,c^2,(abc)^2\rangle$              & $0$             & $2$                   \\
			$(0,1;\{2,2\})$      & $\mathbb{Z}/2\mathbb{Z} * \mathbb{Z}/2\mathbb{Z}$           & $0$             & $2$                   \\
			$(0,2;\emptyset)$    & $\mathbb{Z}$                                                & $0$             & $1$                   \\
			$(1,0;\emptyset)$    & $\mathbb{Z} \times \mathbb{Z}$                              & $0$             & $1$                   \\ \hline
		\end{tabular}
		\caption{The list of non-hyperbolic profinite $F$-groups}\label{fig:Table_nonhyperbolic_F_groups}
	\end{table}
	
	\subsubsection{} In the direction of a more systematic treatment of $m$-step properties of profinite $F$-groups of Section~\ref{sub:Fenchel_ref}, let us give a direct consequence of Theorem~\ref{thm:proFFenchels} regarding the perfectness property -- that is on the triviality of $G^{\rm ab}=G/G^{(1)}$.
	
	\begin{corollary}\label{cor_perfect}
		Any non-trivial open subgroup of a non-perfect profinite $F$-group is non-perfect.
	\end{corollary}
	
	\begin{proof}
		Since for $\Delta$ non-hyperbolic the conclusion follows from Table~\ref{fig:Table_nonhyperbolic_F_groups}, we can assume that $\Delta$ is hyperbolic. We first show that the maximal solvable quotient $\Delta^{\mathrm{solv}}$ of $\Delta$ is an infinite group. Indeed, by Theorem~\ref{thm:classicalproFfenchel}, there exists a torsion-free open normal subgroup $H\triangleleft\Delta$ such that $\Delta/H$ can be assumed to be solvable by \cite{Sah69} Theorem~1.5(c).
		Since $H^{(1)}$ is a characteristic subgroup of $H$, it follows that $H^{(1)}$ is normal in $\Delta$, and we obtain the following commutative diagram with exact rows:
		\begin{equation*}
			\xymatrix{
				1\ar[r]& H\ar[r]\ar@{->>}[d]& \Delta \ar[r] \ar@{->>}[d]& \Delta/H\ar[r]\ar@{=}[d]&1\\
				1\ar[r]& H^{\mathrm{ab}}\ar[r]& \Delta/H^{(1)} \ar[r]& \Delta/H\ar[r]&1\\
			}
		\end{equation*}
		Since $H^{\mathrm{ab}}$ is infinite, $\Delta/H^{(1)}$ is also infinite.
		Additionally $\Delta/H^{(1)}$ is solvable thus $\Delta^{\mathrm{solv}}$ is infinite.
		
		Let $M$ be a non-trivial open subgroup of $\Delta$ and assume that it is perfect.
		Note that if the composite  $M\hookrightarrow \Delta \twoheadrightarrow \Delta^{\mathrm{solv}}$ is non-trivial, then $M$ has a solvable quotient.
		Since perfect groups have no solvable quotient, it thus follows that $M<\ker(\Delta\twoheadrightarrow \Delta^{\mathrm{solv}})$.
		This contradicts the openness hypothesis on $M$, since by the claim, $\ker(\Delta\twoheadrightarrow \Delta^{\mathrm{solv}})$ is not an open subgroup of $\Delta$.
		
	\end{proof}
	
	\subsection{Torsion freeness for commutator subgroups in the proper and $(0,1)$ cases}\label{sub:torsfreecommutator}
	
	\subsubsection{}\label{subsub:abelPer} In the case of a compact profinite $F$-group, that is with no cusp, we relate the orders of the inertia subgroups to the cardinality of the abelianization and the signature of the commutator subgroup in this case.  Our numerical condition for the orbifold inertia groups $I_i\hookrightarrow \Delta$ to survive in the abelianized group is as follows.
	
	\begin{lemma}\label{lem:orderinab}
		Let $\Delta$ be a profinite $F$-group with signature $(g,0;\{n_{1},\dots,n_{k}\})$, and assume that $k \geq 2$. Then the order of an orbifold inertia group $I_{0}=\langle \delta_{i_0}\rangle$ in $\Delta^{\mathrm{ab}}$ is given by $\gcd(n_{i_0}, \operatorname{lcm}(n_j\mid j\neq i_0))$.
		In particular, $n_{i_0}$ divides $\mathrm{lcm}(n_j\mid j\neq i_0)$ if and only if  the natural morphism $I_{i_0} \to \Delta^{\mathrm{ab}}$ is injective.
	\end{lemma}
	
	\begin{proof}
		Let $\psi$ be the natural homomorphism:
		\begin{equation*}
			\psi: \bigoplus_{i=1}^{k} \mathbb{Z} \cdot \delta_i \twoheadrightarrow \left( \bigoplus_{i=1}^{k} \mathbb{Z}/n_i\mathbb{Z} \cdot \delta_i \right) \Big/ \left\langle \sum_{i=1}^{k} \delta_i \right\rangle.
		\end{equation*}
		Then $\ker\psi$ is generated by $(n_1,0,\dots,0),\dots,(0,\dots,0,n_k)$ and $(1,\dots,1)$.
		Consequently, the tuple $(0,\dots,0,b,0,\dots,0)$ is in $\ker\psi$ if and only if the system of congruences $a\equiv b\pmod{n_i}$, $a\equiv 0\pmod{n_j}$, $ j\neq i$, has a solution, that is, if and only if  $\gcd(n_i, \mathrm{lcm}(n_j\mid j\neq i))$ divides $b$.
		
		Since the order of $\delta_{i_0}$ in $\Delta^{\mathrm{ab}}$ equals the smallest positive integer $b$ that satisfies the above required condition, the order of the image of $\delta_{i_0}$ in $\Delta^{\mathrm{ab}}$ is given by $\gcd(n_{i_0}, \mathrm{lcm}(n_j\mid j\neq i_0))$.
		
		\medskip
		
		For the second assertion, note that, for $n_{i_0}$ to divide $\mathrm{lcm}(n_j\mid j\neq i_0)$ is equivalent for $n_{i_0}$ to satisfying the condition $n_{i_0}=\gcd(n_{i_0}, \mathrm{lcm}(n_j\mid j\neq i_0))$, which is equivalent to the order of $\delta_{i_0}$ in $\Delta^{\mathrm{ab}}$ being equal to $n_{i_0}$.
	\end{proof}
	
	Note that following Remark~\ref{rem:McAbGr} below, the condition $k\geq 2$ is necessary for the derived subgroup of $\Delta$ to be geometric, that is, to be realizable as an analytic surface group. In the following, we consider the maximal abelian torsion-free quotient $\Delta_{\rm tf}^{\rm ab}=\Delta^{\mathrm{ab}}/\Delta^{\rm ab}_{\rm tor}$ -- where $\Delta^{\mathrm{ab}}_{\rm tor}$ denotes the closed subgroup generated by all the torsion elements of $\Delta^{\rm ab}$ -- and we write $\Delta_0=\ker(\Delta\to\Delta_{\rm tf}^{\rm ab})$.

	\begin{lemma}\label{lem:inducedsignature}
		Let $\Delta$ be a profinite $F$-group with signature $(g,0;\{n_1,\dots, n_k\})$ and assume $k\geq 2$. Then
		$\Card(\Dab_{\rm tor})=\gcd(\prod_{i\in I} n_i \mid \Card I
		= k-1)$
		and $\Delta_0$ is open in $\Delta$, with induced signature
		\begin{equation*}
			\Sigma_{\Delta_0}=\bigl(g',0;
			\{\underbrace{\omega_1,\dots,\omega_1}_{m_1\ \text{times}},\
			\underbrace{\omega_2,\dots,\omega_2}_{m_2\ \text{times}},\
			\dots,\
			\underbrace{\omega_k,\dots,\omega_k}_{m_k\ \text{times}}\}\bigr),
		\end{equation*}
		where we denote
		\begin{equation*}
			\omega_i\coloneqq \frac{n_i}{\gcd(n_i, \mathrm{lcm}(n_j\mid j\neq i))}, \qquad
			m_i\coloneqq \frac{\Card(\Delta^{\mathrm{ab}}_{\rm tor}
				)}{\gcd(n_i, \mathrm{lcm}(n_j\mid j\neq i))},
		\end{equation*}
		and $g'$ is the unique integer that satisfies the equation of Proposition~\ref{prop:inducedsign}~\ref{RH_formula}.
	\end{lemma}
	
	Note that in the notation $\Sigma_{\Delta_0}$ we omit the indices $i$ for which $\omega_i=1$. Under the assumption $\Sigma_\Delta=(0,0;\{n_1,\dots, n_k\})$, the same statement holds, replacing $\Dab_{\rm tor}$ (resp. $\Delta_0$) by $\Dab$ (resp.~$\DOne$).
	
	\begin{proof}
		By the presentation of $\Delta$ and Eq.~\eqref{eq:DabExpl},  we have $ \Dab_{\rm tor}\cong (\prod_{i=1}^{k} \mathbb{Z}/n_i\mathbb{Z}) / \langle (1,\dots,1) \rangle$ which is finite and implies that $\Delta_0$ is open in $\Delta$. Let us compute its order. Fixing $p$ a rational prime, we compute
		\begin{align*}
			\mathrm{ord}_p\left(
			\mathrm{lcm}(n_1,\dots,n_k) \cdot
			\gcd\left(\prod_{i\in I} n_i \middle| \Card I = k-1\right)
			\right)
			& = \mathrm{ord}_p(n_k) + \sum_{i=1}^{k-1} \mathrm{ord}_p(n_i) \\
			& = \mathrm{ord}_p\left(\prod_{i=1}^{k} n_i\right).
		\end{align*}
		Running over all rational primes,  we then obtain
		\begin{equation}\label{eq:htshyhdh}
			\mathrm{lcm}(n_1,\dots,n_k) \cdot \gcd\left(\prod_{i\in I} n_i \middle| \Card I = k-1\right)
			= \prod_{i=1}^k n_i,
		\end{equation}
		which in turn implies
		\begin{equation*}
			\Card( \Dab_{\rm tor})
			= \frac{\prod_{i=1}^k n_i}{\mathrm{lcm}(n_1,\dots,n_k)}
			= \gcd\left(\prod_{i\in I} n_i \middle| \Card I
			= k-1\right).
		\end{equation*}
		
		Since $\Dab_{\rm tor}$ is abelian, we have that every disjoint cycle of the permutation induced by $\delta_i$ on $\Dab_{\rm tor}$ has the same order  $\gcd(n_i, \mathrm{lcm}(n_j\mid j\neq i))$, which is the order of $\delta_{i}$ in $\Dab_{\rm tor} $ by Lemma~\ref{lem:orderinab}.
		By Proposition~\ref{prop:inducedsign} the induced signature of $\Delta_0$ is thus given by
		\begin{equation*}
			\bigl(g',0;
			\{\underbrace{\omega_1,\dots,\omega_1}_{m_1\ \text{times}},\
			\underbrace{\omega_2,\dots,\omega_2}_{m_2\ \text{times}},\
			\dots,\
			\underbrace{\omega_k,\dots,\omega_k}_{m_k\ \text{times}}\}\bigr).
		\end{equation*}
	\end{proof}
	
	The two previous lemmata motivate the introduction of the following condition on the periods of orbifold points.
	\begin{definition}\label{def:LCM}
		Let $\Delta$ be a profinite $F$-group with signature $(g, r; \{n_1, \dots, n_k\})$ and $k\geq 2$.
		We say that $\Delta$ satisfies the \textit{abelian-period condition} if
		\begin{equation*}
			\forall i\in\{1,\dots,k\},\qquad n_i\ \text{ divides }\ \mathrm{lcm}(n_j\mid j\neq i).
		\end{equation*}
	\end{definition}
	
	By Proposition~\ref{prop:SigIsoInv}, this condition does not depend on the choice of the signature of $\Delta$. As a convention, in the case $k=0$ (resp. $k=1$), we say that $\Delta$ satisfies (resp. does not satisfy) the abelian-period condition.
	
	\begin{remark}\label{rem:McAbGr}
		For a discrete $F$-group the \emph{abelian period condition} is nothing else than the L.C.M. condition of \cite{Mac65} which guarantees that a given abelian group, in particular the derived group of a discrete $F$-group, can be realized as the automorphism group of an analytic surface.
	\end{remark}

	We show that, in the case of $F$-groups with no cusps, the \emph{abelian-period condition} characterizes the torsion-freeness of commutator subgroups.
	
	\begin{proposition}\label{prop:LCMcondition}
		Let $\Delta$ be a profinite $F$-group with signature $(g, 0; \{n_1, \dots, n_k\})$.
		Then the following hold:
		\begin{enumerate}
			\item\label{prop:LCMcondition1}
			The group $\Delta$ satisfies the abelian-period condition if and only if $\Delta_0=\ker(\Delta\to \Dab_{\rm tor})$ is a surface group.
			\item\label{prop:LCMcondition2}
			The periods $n_1,\dots,n_k$ satisfy the condition
			\begin{equation}\tag{$\ast$}
				\forall i\in\{1,\dots,k\}, \text{ either }
				\Bigl(n_i\mid \mathrm{lcm}(n_j\mid j\neq i)\Bigr) \text{ or }
				\Bigl(\ \frac{\prod_{j\neq i} n_j}{\mathrm{lcm}(n_j\mid j\neq i)}\geq 2\ \Bigr)
			\end{equation}
			if and only if $\Delta_0$ satisfies the abelian-period condition.
		\end{enumerate}
	\end{proposition}
	
	In summary, the following implications can be drawn:
	\begin{equation*}
		\left(
		\begin{array}{ccccccc}
			\Delta:\text{ abelian-period}   & \Leftrightarrow & \Delta_0:\text{ torsion-free} \\
			\Downarrow              &                 & \Downarrow   \\
			\Delta: (\ast) & \Leftrightarrow & \Delta_0:\text{ abelian-period}       \\
		\end{array}\right)
	\end{equation*}
	
	\begin{proof}
		If $\Delta$ is perfect, then $\Delta=\ker(\Delta\to \Dab_{\rm tor})$ is not torsion-free and the condition $(\ast)$ is not satisfied. In particular, both sides of the conditions in the statements \cref{prop:LCMcondition1,prop:LCMcondition2} are always false. In the rest of the proof, we may thus assume that $\Delta$ is non-perfect.
		
		\medskip
		
		\noindent \cref{prop:LCMcondition1}
		By Lemma~\ref{lem:orderinab}, the abelian-period condition is equivalent to the fact that the natural morphism $\langle\delta_{i}\rangle\to \Delta^{\mathrm{ab}}$ is injective for any $i \in \{1, \dots, k\}$.
		Hence $\Delta$ satisfies the abelian-period condition if and only if $\DOne \cap \langle \delta_i \rangle = \{1\}$ for any $i \in \{1, \dots, k\}$. Since all torsion elements are contained in one of the conjugates of $\langle \delta_i \rangle$, this is in turn equivalent to $\DOne$ being torsion-free. Finally, by Lemma~\ref{lem:inducedsignature}, $\Delta_0$ is open and torsion-free and is thus a surface group, since it also has no cusp by Proposition~\ref{prop:inducedsign}.
		
		\medskip
		
		\noindent\cref{prop:LCMcondition2}
		By Lemma~\ref{lem:inducedsignature}, $\Delta_0$ is an open subgroup of $\Delta$. Given the signature $\Sigma_{\Delta_0}$ induced from $\Delta$ as ibid. the condition $(\ast)$ is thus equivalent to the condition that, for any $i \in \{1, \dots, k\}$, either $\omega_i=1$ or $m_i\ge 2$. Thus, if condition $(\ast)$ is satisfied, then $\Delta_0$ satisfies the abelian-period condition.
		
		Let us show the converse implication, and assume that the periods $n_1,\dots, n_k$ do not satisfy condition $(\ast)$.
		Then there exists $i$ such that $\omega_i\ge 2$ and $m_i=1$.
		Let $p$ be a prime number such that $v_{p}(\omega_{i})\geq 1$, that is $v_p(n_i)> \max\{v_p(n_j)\mid j\neq i\}$.
		Then, for any $j\neq i$
		\[
		v_p(\omega_j)=\max\{0,\,v_p(n_j)-\max\{v_p(n_\ell)\mid \ell\neq j\}\} = 0,
		\]
		which implies that $\omega_i$ doesn't divide $\mathrm{lcm}(\omega_{j}\mid j\neq i)$.
		Since $m_{i}=1$, this is exactly that $\Delta_0$ does not satisfy the abelian-period condition.
	\end{proof}

	\begin{example}\label{ex:proFFenchels}
		Consider a profinite $F$-group such that $(g,r)=(0,0)$ with periods given by $\{n_1, n_2, n_3\} = \{a_1, a_2a_3, a_2a_4\}$ for some positive integers $a_1,\dots,a_4$ coprime to each other.
		These periods don't satisfy the condition $(\ast)$, since we have $a_2a_3\nmid \mathrm{lcm}(a_1,a_2a_4)=a_1a_2a_4$ and $a_1\cdot(a_2a_4)/\mathrm{lcm}(a_1,a_2a_4)=1<2$.
		
		\medskip
		
		Let us compute the signatures of $\DOne$ and $\Delta^{(2)}$.
		\begin{enumerate}[(i)]
			\item For $\DOne$, by gcd computation, one obtains $\Delta^{\mathrm{ab}}\cong\mathbb{Z}/a_2\mathbb{Z}$, and finally
			\begin{equation*}
				\omega_1=a_1,\quad
				\omega_2=a_3,\quad
				\omega_3=a_4 \text{ and } 
				m_1=a_2,\quad
				m_2=1,\quad
				m_3=1.
			\end{equation*}
			Therefore $\DOne$ has signature
			\begin{equation*}
				\Sigma_{\DOne}=(0,0;\{\underbrace{a_1,\dots,a_1}_{a_2\ \text{times}},\ a_3,a_4\});
			\end{equation*}
			and its periods satisfy only the condition $(\ast)$,  since $a_3\nmid a_1a_4$ and $
			a_1^{a_2}a_4/\mathrm{lcm}(a_1,a_4)=a_1^{a_2-1}\geq 2$.
			\item By a similar computation, the group $(\DOne)^{(1)}$ has signature
			\begin{equation*}
				\left(1+\frac{a_1^{a_2-2}}{2}\Bigl(a_2(a_1-1)-2a_1\Bigr),0;\{\underbrace{a_3,\dots,a_3}_{a_1^{a_2-1}\ \text{times}},\ \underbrace{a_4,\dots,a_4}_{a_1^{a_2-1}\ \text{times}}\}\right)
			\end{equation*}
			whose periods satisfy the abelian-period condition.
		\end{enumerate}
	\end{example}
	
	{ 
		\subsubsection{}\label{subsub:zerOne}  
		The following lemma explains why, in the $m$-step anabelian reconstruction, the case of  a profinite $F$-group $\Delta$ with a signature such that $(g,r)=(0,1)$ has to be handled specially: for any torsion-free open normal subgroup $H\triangleleft \Delta$ -- that would correspond to a finite cover of the curve -- the number of cusps $r_H$ is bounded above, by $r_{\DOne}$. Note that one can find \emph{hyperbolic} examples with $r_{\DOne}=1.$
		
		\begin{lemma}\label{cor:r_greater}
			Let $\Delta$ be a non-perfect profinite $F$-group with signature $(0, 1; \{n_1, \dots, n_k\})$. The commutator subgroup $\DOne$ is a torsion-free open normal subgroup of $\Delta$ with induced signature given by
			\begin{equation*}
				g_{\DOne}=\frac12\Bigl(-\chi(\Delta)\operatorname{Card}(\Delta^{\mathrm{ab}})+2-r_{\DOne}\Bigr)
				\text{ and }
				r_{\DOne}=\frac{\prod_{i=1}^k n_i}{\operatorname{lcm}(n_1,\dots,n_k)}.
			\end{equation*}
			Furthermore, if $\Delta$ is hyperbolic, then $g_{\DOne}\geq 1$.
		\end{lemma}
		
		\begin{proof}
			As previously, the presentation \cref{equation:01-decomposition} implies:
			\begin{align*}
				\Delta^{\mathrm{ab}}
				& \cong \left(\bigoplus_{i=1}^k \mathbb{Z}/n_i \mathbb{Z} \cdot \delta_i\right)
			\end{align*}
			and it follows that $\DOne$ is a torsion-free open normal subgroup of $\Delta$. By this isomorphism, $\gamma_{1}$ is mapped to the element $(-1,\cdots,-1)$. The permutation action of $\gamma_1$ on $\Delta^{\mathrm{ab}}$ is thus given by disjoint cycles of length $\operatorname{ord} (-1,\dots, -1)$. We now obtain from Proposition~\ref{prop:inducedsign} that
			\begin{equation*}
				r_{\DOne}=\frac{\Card(\Delta^{\mathrm{ab}})}{\mathrm{ord}((-1,\cdots,-1))}=\frac{\prod_{i=1}^k n_i}{\mathrm{lcm}(n_1,\dots,n_k)}.
			\end{equation*}
			By the Riemann-Hurwitz formula for profinite $F$-groups we now have
			\begin{equation*}
				g_{\DOne}=\frac12\Bigl(-\chi(\Delta)\operatorname{Card}(\Delta^{\mathrm{ab}})+2-r_{\DOne}\Bigr).
			\end{equation*}
			This completes the proof of the first assertion.
			
			\medskip
			
			Next, we assume that $\Delta$ is hyperbolic.
			We have
			\begin{equation*}
				-\chi(\Delta)=\Bigl(\sum_{i=1}^k\bigl(1-\tfrac1{n_i}\bigr)\Bigr)-1
				=(k-1)-\sum_{i=1}^k\frac1{n_i}.
			\end{equation*}
			Multiplying both sides by $\mathrm{lcm}(n_1,\dots,n_k)$, this implies that
			\begin{equation*}
				-\chi(\Delta)\cdot\mathrm{lcm}(n_1,\dots,n_k)
				=(k-1)\cdot \mathrm{lcm}(n_1,\dots,n_k)-\sum_{i=1}^k \frac{\mathrm{lcm}(n_1,\dots,n_k)}{n_i}.
			\end{equation*}
			Here, the left-hand side is positive by the hyperbolicity hypothesis, and  the right-hand side is an integer immediately.
			These imply that
			\begin{equation*}
				-\chi(\Delta)\cdot\mathrm{lcm}(n_1,\dots,n_k)\geq 1.
			\end{equation*}
			Therefore we get
			\begin{equation*}
				-\chi(\Delta)\cdot\operatorname{Card}(\Delta^{\mathrm{ab}})=-\chi(\Delta)\cdot\operatorname{lcm}(n_1,\dots,n_k)\cdot \frac{\prod_{i=1}^k n_i}{\operatorname{lcm}(n_1,\dots,n_k)}\geq \frac{\prod_{i=1}^k n_i}{\operatorname{lcm}(n_1,\dots,n_k)}\ =r_{\Delta^{(1)}}.
			\end{equation*}
			Thus, by the first assertion, we have
			\begin{equation*}
				g_{\DOne}=\frac12\Bigl(-\chi(\Delta)\operatorname{Card}(\Dab)+2-r_{\DOne}\Bigr)\geq \frac12\Bigl(r_{\DOne}+2-r_{\DOne}\Bigr)=1.
			\end{equation*}
		\end{proof}
	}
	From the orbifold inertia point of view, on the other hand, $\Dab$ still contains all the orbifold inertia groups, and since $\Dab\cong \oplus_{i=1}^k \mathbb{Z}/n_i \mathbb{Z}$, the subgroup $\DOne$ is also the \emph{minimal torsion-free} open normal subgroup of $\Delta$ whose quotient is abelian.
	
	\subsection{Non-perfect Fenchel-Nielsen and applications}\label{sub:Fenchel_ref}
	We establish some refinements of the profinite Fenchel-Nielsen Theorem from the perspective of solvable quotients of bounded derived length. For a profinite group $G$, we recall that $G^{(m)}$ denotes its topological $m$-derived subgroup, and $G^{m} \coloneqq G / G^{(m)}$ its \emph{maximal $m$-step solvable quotient}.
	
	\subsubsection{} 
	The following provides a combinatorial group theoretic criterion for perfectness. Its proof relies on a numerical lemma, which is stated hereafter.
	
	\begin{proposition}\label{prop:charperfect}
		Let $\Delta$ be a nontrivial profinite $F$-group with signature $(g, r; \{n_1, \dots, n_k\})$.
		Then $\Delta$ is perfect if and only if $(g,r)=(0,0)$ and $\gcd(n_i, n_j) = 1$ for any $i \neq j$.
	\end{proposition}
	
	A similar statement holds in the discrete case, which relies on some Fuchsian group arguments, see \cite{Sah69} Theorem~1.5~(b); the following proof is purely group-theoretic and in terms of the fundamental group.
	\begin{proof}
		First, assume that $(g,r)=(0,0)$ and that $\gcd(n_i, n_j) = 1$ for any $i \neq j$. Since $\Delta$ is nontrivial, we must have $k\geq 2$.	Under these hypotheses and by the presentation \cref{equation:01-decomposition}, we have
		\begin{equation}\label{eq:grgarghteg}
			\Delta^{\mathrm{ab}} \cong \left(\prod_{i=1}^{k} \mathbb{Z}/n_i\mathbb{Z}\right) / \langle (1,\dots,1) \rangle.
		\end{equation}
		By the Chinese remainder theorem, since $\gcd(n_i, n_j) = 1$ for all $i \neq j$, the group $\left(\prod_{i=1}^{k}\mathbb{Z}/n_{i}\mathbb{Z}\right)$ is cyclic and generated by the element $(1,\dots,1)$. Hence, the group $\Delta^{\mathrm{ab}}$ is trivial in this case.
		
		\medskip
		
		Next, we show the converse implication. Let us assume that $\Delta^{\mathrm{ab}}$ is trivial. By the presentation \cref{equation:01-decomposition}, we have the following isomorphisms depending on whether $r=0$ or $r\geq 1$
		\begin{align*}
			\Delta^{\mathrm{ab}} \cong\; &
			\bigoplus_{i=1}^g \bigl(\widehat{\mathbb{Z}}\alpha_i \oplus \widehat{\mathbb{Z}}\beta_i \bigr)
			\;\oplus\; \bigoplus_{i=1}^{r-1} \widehat{\mathbb{Z}}\, \gamma_i
			\;\oplus\; \left(\bigoplus_{i=1}^k (\mathbb{Z}/n_i\mathbb{Z}) \cdot \delta_i\right)
			\qquad                       & \text{(when $r \geq 1$)} \\
			\Delta^{\mathrm{ab}} \cong\; &
			\bigoplus_{i=1}^g \bigl(\widehat{\mathbb{Z}}\alpha_i \oplus \widehat{\mathbb{Z}}\beta_i \bigr)
			\;\oplus\; \left(\prod_{i=1}^{k} \mathbb{Z}/n_i\mathbb{Z}\right) / \langle (1,\dots,1) \rangle.
			\qquad                       & \text{(when $r = 0$)}
		\end{align*}
		The fact that $\Delta^{\mathrm{ab}}=\{1\}$ thus gives that $g=0$ in both cases. Moreover, if $r\geq 1$ the signature has to be $(0,1;\emptyset)$ and $\Delta$ itself is trivial in that case, a contradiction.
		We thus have  $(g,r)=(0,0)$. If $k \leq 1$, then $\Delta$ is again trivial, and so we may further assume that $k \geq 2$.
		By Lemma~\ref{lem:inducedsignature}, we get that
		\begin{equation*}
			1=\Card(\Delta^{\mathrm{ab}}) = \gcd\left(\prod_{i\in I} n_i \middle| \Card I = k-1\right).
		\end{equation*}
		We now remark that $\gcd\left(\prod_{i\in I} n_i \middle| \Card I = s\right)$ divides $\gcd\left(\prod_{i\in I} n_i \middle| \Card I = k-1\right)$ for any $s \in \{1, \dots, k-1\}$ so that the former is again equal to $1$. By Lemma~\ref{rgabagwra}, we conclude that $\gcd(n_i, n_j) = 1$ for all $i \neq j$.
	\end{proof}

	\begin{lemma}\label{rgabagwra}
		For $n_{1}, \dots, n_{k}$ positive integers, one has
		\begin{equation*}
			\prod_{1 \leq i < j \leq k} \gcd(n_i, n_j) = \prod_{s = 1}^{k - 1} \gcd\left(\prod_{i \in I} n_i \middle| \Card I = s\right)
		\end{equation*}
	\end{lemma}
	
	\begin{proof}
		Let us fix $\ell$ a prime number and establish the equality for the $\ell$-adic valuation. Up to the  order of the $n_i$s, we may assume that
		$v_\ell(n_1) \leq v_\ell(n_2) \leq \cdots \leq v_\ell(n_k)$, so the LHS becomes
		\begin{equation*}
			v_{\ell}\left(\prod_{1\leq i<j\leq k}\gcd(n_{i},n_{j})\right)=\sum_{1\leq i<j\leq k}v_{\ell}(n_{i})=\sum_{i=1}^{k}(k-i)v_{\ell}(n_{i}).
		\end{equation*}
		Denoting $d_s = \gcd(\prod_{i \in I} n_i\mid\Card I = s )$ for $s =1, \dots, k$, the RHS gives:
		\begin{equation*}
			v_{\ell}\left(\prod_{s=1}^{k-1}d_{s}\right)=\sum_{s=1}^{k-1}v_{\ell}(d_{s})=\sum_{s=1}^{k-1}\left(\sum_{i=1}^{s}v_{\ell}(n_{i})\right)=\sum_{i=1}^{k}(k-i)v_{\ell}(n_{i}).
		\end{equation*}
		The equality then follows by running over every prime number $\ell$.
	\end{proof}
	
	\subsubsection{}
	We state a $m$-step refinement of the profinite Fenchel-Nielsen Theorem~\ref{thm:classicalproFfenchel} in terms of derived conditions. 
	The following lemma is certainly well-known by experts.
	
	\begin{lemma} \label{lem:charsubgroups}
		Let $\Delta$ be a profinite $F$-group and $m$ a positive integer. Let $U< \Delta$ be a torsion-free open normal subgroup such that $\Delta/U$ has derived length at most $m$. Then there is a torsion-free characteristic open subgroup $\overline{U}< \Delta$ such that $\Delta/\overline{U}$ has derived length at most $m$. 
	\end{lemma}
	
	\begin{proof}
		Let $U$ be a torsion-free open normal subgroup of $\Delta$. As $\Delta$ is finitely generated, there are only finitely many open subgroups of $\Delta$ with a fixed index (see \cite[Proposition 2.5.1(a)]{MR2599132}). Hence $\overline{U}=\bigcap_{\varphi\in\mathrm{Aut}(\Delta)}\varphi(U)=\bigcap_{i=1}^n U_i$ is an open subgroup of $\Delta$ which is torsion-free and characteristic by construction. The quotient $\Delta/\overline{U}$ embeds diagonally into the finite direct product of the groups $\prod_{i=1}^n\Delta/U_i$ so that it has derived length at most $m$ since each of the factors has.
	\end{proof}
	
	In terms of anabelian reconstruction as in Section~\ref{subsub:rigidification}, and more generally as in Section~\ref{subsection:2.2}, the following result provides some group-theoretic criterion for the orbifold inertia to survive in the derived group $\Delta^{(2)}$ (resp. $\Delta^{(1)}$), that is, in the $m$-step maximal quotients $\Delta^{2}$ (resp. $\Delta^{1}$).
	
	\begin{theorem}\label{thm:proFFenchels}
		Any non-perfect profinite $F$-group $\Delta$ contains a torsion-free open characteristic subgroup whose quotient has derived length at most $3$. We can furthermore choose such a subgroup with a quotient that is
		\begin{enumerate}
			\item\label{thm:proFFenchels--1} abelian if and only if, either $r \geq 1$, or $r=0$ and $\Delta$ satisfies the abelian-period condition.
			\item\label{thm:proFFenchels-3} of derived length at most $2$ if and only if, either $(g,r)\neq (0,0)$, or $(g,r)=(0,0)$ and $\DOne$ satisfies the abelian-period condition.
		\end{enumerate}
		Each of the previous numerical conditions above holds if one assumes the subgroup to be only normal.
	\end{theorem}
	
	Note that we will only have use of normal subgroups in the rest of the text. 
	
	\begin{proof}
		By Lemma~\ref{lem:charsubgroups} it is enough to show that open normal subgroups with the corresponding properties exist. 
		
		\medskip
		
		\noindent\cref{thm:proFFenchels--1}
		Since $\Delta^{\rm ab}_{\rm tor}$ is finite by the presentation \cref{equation:01-decomposition},  $\ker(\Delta\twoheadrightarrow \Delta^{\mathrm{ab}}_\mathrm{tor})$ is an open normal subgroup of $\Delta$ whose quotient is abelian.
		When  $r=0$, the equivalence follows from Proposition~\ref{prop:LCMcondition}~\cref{prop:LCMcondition1}.
		When $r\geq 1$, we have the following isomorphism:
		\begin{align*}
			\Delta^{\mathrm{ab}}
			& \cong \bigoplus_{i=1}^g \left(\hat{\mathbb{Z}}\alpha_i \oplus \hat{\mathbb{Z}}\beta_i \right)
			\oplus \bigoplus_{i=1}^{r-1} \hat{\mathbb{Z}} \gamma_i
			\oplus \left(\bigoplus_{i=1}^k \mathbb{Z}/n_i \mathbb{Z} \cdot \delta_i\right)
		\end{align*}
		Hence $\ker(\Delta\twoheadrightarrow \Delta^\mathrm{ab}_\mathrm{tor})$ is  torsion-free.
		Therefore, the condition is always true when $r\geq 1$.
		
		\medskip
		
		\noindent\cref{thm:proFFenchels-3}
		By \cref{thm:proFFenchels--1}, we may assume that $r=0$.
		First, we show the case where $(g,r)=(0,0)$. In this case, we have that $\DOne$ is open in $\Delta$ (see Lemma~\ref{lem:inducedsignature}).
		Hence, by \cref{thm:proFFenchels--1},  $\DOne$ satisfies the abelian-period condition if and only if there exists a torsion-free open characteristic subgroup of $\DOne$ whose quotient is abelian. Thus the equivalence, which completes the proof of the case where $(g,r)=(0,0)$ by Lemma~\ref{lem:charsubgroups}.
		
		Next, let us consider the case where $g\geq 1$ and show that $\Delta$ always satisfies the quotient condition. The maximal torsion-free quotient $\Dtf$ of $\Delta$ is a profinite $F$-group with signature $(g,r;\emptyset)$.
		For $n \in \mathbb{Z}_{\geq 2}$ let us consider the kernel $K$ of the natural surjection $\Delta\to(\Dtf)^{\mathrm{ab}}/n$, which is a characteristic open subgroup and contains all the torsion subgroups of $\Delta$.
		By Proposition~\ref{prop:inducedsign}, the induced signature of $K$ is given by
		\begin{equation*}
			\Sigma_K=(g_K, 0; \{\underbrace{n_1, \dots, n_1}_{p\text{ times}},\
			\underbrace{n_2, \dots, n_2}_{p\text{ times}},\
			\dots,\
			\underbrace{n_k, \dots, n_k}_{p\text{ times}}\}), \text{ with } p\coloneqq [\Delta:K].
		\end{equation*}
		
		Since $g\geq 1$, we have that $p> 1$, and thus $K$ satisfies the abelian-period condition (Note that $K$ is non-perfect by Proposition~\ref{prop:charperfect}.)
		Therefore, by~\ref{thm:proFFenchels--1}, there exists a torsion-free characteristic open subgroup $H$ of $K$ such that $K/H$ is abelian.
		As $\Delta/K$ is abelian, $\Delta/H$ has derived length at most $2$.
		
		\medskip
		
		\noindent
		We now prove the statement in the general case. From \cref{thm:proFFenchels-3}, we may assume that $(g,r)=(0,0)$. We may also assume that $k\geq3$. Since $\Delta$ is non-perfect, there exist distinct indices $i \ne j$ such that $\gcd(n_i,n_j)\ne 1$ by Proposition~\ref{prop:charperfect}. After reordering the indices if necessary, we may assume that $\gcd(n_1,n_2) \ne 1$. Let $\ell$ be a prime number dividing $\gcd(n_1, n_2)$.
		Define the following surjective morphism:
		\begin{equation*}
			\psi: \Delta \to \mathbb{Z}/\ell\mathbb{Z},\quad \delta_1 \mapsto 1,\quad \delta_2 \mapsto -1,\quad \delta_i \mapsto 0\text{ for }i \ne 1,2.
		\end{equation*}
		Since  $\ker(\psi)$ is an open normal subgroup of $\Delta$, it is also a profinite $F$-group whose induced signature is
		\begin{equation*}
			(0, 0; \{\tfrac{n_1}{\ell}, \tfrac{n_2}{\ell}, \underbrace{n_3, \dots, n_3}_{\ell\ \text{times}}, \dots, \underbrace{n_k, \dots, n_k}_{\ell\ \text{times}}\})
		\end{equation*}
		by Proposition~\ref{prop:inducedsign}, where $\tfrac{n_1}{\ell}$ and $\tfrac{n_2}{\ell}$ each appear only once, and each $n_i$ is repeated $\ell$ times for $i \geq 3$.
		Immediately, the induced signature of $\ker(\psi)$ satisfies the condition $(\ast)$ of Proposition~\ref{prop:LCMcondition}~\cref{prop:LCMcondition2}, and hence $(\ker(\psi))^{(1)}$ satisfies the abelian-period condition.
		Therefore, \cref{thm:proFFenchels-3} implies that there exists a torsion-free open characteristic subgroup $H$ of $\ker(\psi)$ whose quotient has derived length at most $2$.
		Hence we obtain a torsion-free open normal subgroup of $\Delta$ whose quotient has derived length at most $3$.
	\end{proof}
	
	\subsubsection{}
	As a consequence, we can control the growth of cuspidal data in the derived quotients -- we further refer to Section~\ref{subsub:zerOne} and Lemma herein for a discussion on the role of the case $(g,r)=(0,1)$.
	
	\begin{corollary}\label{cor:r_greater-1}
		Let $\Delta$ be a non-perfect hyperbolic profinite $F$-group.
		If $r\geq 1$ and $(g,r)\neq (0,1)$, resp. if $(g,r)= (0,1)$, then for any integer $r_0$ there exists a torsion-free open characteristic subgroup $H$ of $\Delta$ such that the quotient $\Delta/H$ is abelian (resp. has derived length at most $2$) and that the number $r_{H}$ of cusps of $H$ satisfies $r_{H}\geq r_{0}$.
	\end{corollary}
	
	\begin{proof}
		Recall that by the presentation \cref{equation:01-decomposition}, we have:
		\begin{align*}
			\Delta^{\mathrm{ab}}
			& \cong \bigoplus_{i=1}^g \left(\hat{\mathbb{Z}}\alpha_i \oplus \hat{\mathbb{Z}}\beta_i \right)
			\oplus \bigoplus_{i=1}^{r-1} \hat{\mathbb{Z}} \gamma_i
			\oplus \left(\bigoplus_{i=1}^k \mathbb{Z}/n_i \mathbb{Z} \cdot \delta_i\right)
		\end{align*}
		When $(g,r)\neq (0,1)$, for any integer $n_{0}$, there exists an open subgroup $\tilde{H}$ of $\Delta^{\mathrm{ab}} $ such that $\tilde{H}$ is torsion-free and that $[\Delta^{\mathrm{ab}}:\tilde{H}]\geq n_0$.
		By taking $n_{0}$ large enough, and as a consequence of the hyperbolicity when $(g,r)=(0,2)$, the inverse image of $\tilde{H}$ in $\Delta$ satisfies the desired conditions. 
		Next we assume that $(g,r)=(0,1)$.
		Then by Theorem~\ref{thm:proFFenchels}~\ref{thm:proFFenchels--1}, $\Delta$ contains a torsion-free open characteristic subgroup $H'$ such that the quotient $\Delta/H'$ is abelian.
		Since we assume that $\Delta$ is hyperbolic, $H'$ is also hyperbolic.
		Hence, by \cite{MR4745885} Lemma~1.5~(2)~(3), there exists an open characteristic subgroup $H$ of $H'$ such that the quotient $H'/H$ is abelian and $r_{H}\geq r_{0}$.
		By replacing $H$ with the characteristic core of $H$ in $\Delta$ if necessary, the assertion follows.
	\end{proof}
	
	\begin{remark}
		An analogous statement, which is not used in this paper, also holds with respect to $g$ instead of $r$ by applying \cite{MR4745885} Lemma~1.5~(1)~(3).
	\end{remark}
	
	\subsubsection{}\label{subsub:SlimCfree} 
	We now deal, in terms of anabelian group properties and of  $m$-step properties, with the center-freeness property of profinite $F$-groups. Recall that a hyperbolic profinite group is said to be slim if the center $Z_H(G)$ is trivial for every open subgroup $H<G$ -- see \cite{Moc04}.
	\begin{proposition}\label{prop:Zorbi}
		A hyperbolic profinite $F$-group is slim. In particular, the center of a hyperbolic profinite $F$-group is trivial.
	\end{proposition}
	
	This follows, after passing to a torsion-free open normal subgroup given by the profinite Fenchel--Nielsen Theorem~\ref{thm:classicalproFfenchel}, by the same arguments of \cite{Moc04} Lemma~1.3.1, which relies on the property that, for a hyperbolic curve $X$ over a field $k$, the natural action
	\[
	\operatorname{Aut}_k(X)\to \Aut (H^{1}_{\rm et}(X,\mathbb{Q}_{\ell})),
	\]
	where $\ell$ is a prime number different from the characteristic of $k$, is faithful -- We also refer to \cite{Nak94} Corollary~1.3.4 for an alternate proof in terms of the goodness of the orbifold fundamental group.

	\medskip
	
	For the reader's convenience, let us recall the following lemma that will be used multiple times -- see \cite{MR4745885} Lemma~1.1~(i).
	
	\begin{lemma}\label{lemma:mstep}
		Let $m,n\geq 0$ be integers.
		Let $G$ be a profinite group, and let $H$ be an open subgroup of $G^{m+n}$ containing $G^{(m)}/G^{(m+n)}$.
		Let $\tilde{H}$ be the inverse image of $H$ in $G$ under the natural surjection $G\twoheadrightarrow G^{m+n}$.
		Then the natural surjection $\tilde{H}^{n}\twoheadrightarrow H^{n}$ is an isomorphism.
	\end{lemma}

	The following profinite $F$-group result extends a previous similar result for the maximal $m$-step solvable quotient of free profinite groups (see \cite{MR4578639} Corollary~1.1.7). 
	
	\begin{proposition}\label{prop:center-free}
		Let $\Delta$ be a hyperbolic profinite $F$-group and denote by $m(\Delta)$ the minimal integer $k$ such that there exists a torsion-free open normal subgroup $H$ that contains $\DerivSeries{\Delta}{k}$. Then $\FinStepSolvQuo{\Delta}{m}$ is center-free for any $m\geq m(\Delta)+2$.
	\end{proposition}
	Special cases of this result include the cases when $\Delta$ is torsion-free (with $m(\Delta)=0$), when $\Delta$ is non-perfect (with $m(\Delta)\in\{0,1,2,3\}$) by Theorem~\ref{thm:proFFenchels}, and when $\Delta$ is affine (with $m(\Delta)\in\{0,1\}$).

	\begin{proof}
		We  assume that $\Delta$ is non-perfect, since $\FinStepSolvQuo{\Delta}{m}$ is trivial otherwise. By the hypothesis $m-1>m(\Delta)$, there exists a torsion-free open normal subgroup $H$ of $\Delta$ that contains $\DerivSeries{\Delta}{m-1}$. Denote by $\overline{H}$ the image of $H$ in $\FinStepSolvQuo{\Delta}{m}$. By Lemma~\ref{lemma:mstep}, we have that $\Delta/H\cong \FinStepSolvQuo{\Delta}{m}/\overline{H}$ and $H^{\ab}\cong \overline{H}^{\ab}$. Under the standard dual $\ell$-adic cohomology identification and the same argument as above, one deduces that the conjugation action of $\FinStepSolvQuo{\Delta}{m}/\overline{H}$ on $\overline{H}^{\ab}$ is faithful, so that:
		\begin{equation}\label{eq1:prop_center-free_newproof}
			\CenterSubgrp{\FinStepSolvQuo{\Delta}{m}}\lesssim \bigcap_{H} \overline{H} =\DerivSeries{\Delta}{m-1}/\DerivSeries{\Delta}{m}.
		\end{equation}

		On the other hand, by the hypothesis $m-2\geq m(\Delta)$ there exists a torsion-free open normal subgroup $H'$ of $\Delta$ that contains $\DerivSeries{\Delta}{m-2}$, so that $\DerivSeries{\Delta}{m-1}< \DerivSeries{H'}{1}$. By applying \cite{Yam26} Lemma~2.3(4) to $K=\DerivSeries{\Delta}{m-1}$, we obtain that the conjugation action
		\begin{equation*}
			H'/ \DerivSeries{\Delta}{m-1}\to \Aut((\DerivSeries{\Delta}{m-1})^{\ab}),
		\end{equation*}
		where $(\DerivSeries{\Delta}{m-1})^{\ab}=\DerivSeries{\Delta}{m-1}/ \DerivSeries{\Delta}{m}$, has no nontrivial fixed points. One thus obtains
		\begin{equation}\label{eq2:prop_center-free_newproof}
			\CenterSubgrp{\FinStepSolvQuo{\Delta}{m}}\cap (\DerivSeries{\Delta}{m-1}/ \DerivSeries{\Delta}{m}) = (\DerivSeries{\Delta}{m-1}/ \DerivSeries{\Delta}{m})^{\FinStepSolvQuo{\Delta}{m-1}}\lesssim (\DerivSeries{\Delta}{m-1}/ \DerivSeries{\Delta}{m})^{(H'/ \DerivSeries{\Delta}{m-1})}=\{1\},
		\end{equation}
		and by  Eq.~\eqref{eq1:prop_center-free_newproof} and Eq.~\eqref{eq2:prop_center-free_newproof}, we conclude that $\CenterSubgrp{\FinStepSolvQuo{\Delta}{m}}=\{1\}$.
	\end{proof}

	\subsection{Affineness and hyperbolicity as $3$-step properties for non-perfect groups}\label{sec:weaklyperfect-hyperbolic}
	Since \emph{hyperbolic profinite $F$-groups are not solvable} -- by the profinite Fenchel--Nielsen theorem and the fact that torsion-free hyperbolic profinite $F$-groups are not solvable -- we formalize some finer notions of perfectness for maximal $m$-step solvable quotients that potentially keep track of the stack inertia according to the abelian-period condition. In this direction, one notices that the notion of $m$-derived perfectness stands as an obstruction to checking if the orbifold inertia survives in the maximal $(m+1)$-step solvable quotient $G^{m+1}$.
	
	\medskip
	
	We establish our first group-theoretic characterization of hyperbolicity and affineness for a profinite $F$-group within the $3$-step solvable quotients.

	\subsubsection{}\label{subsub:mStpSol}
	Accordingly, our first refined definitions of perfectness is as follows.
	
	\begin{definition}\label{def:solweak}
		A profinite group $G$ is said to be
		\begin{enumerate}
			\item \emph{$m$-step solvable} if its derived length is at most $m$. 
			\item\label{def:weaklyperfect} \emph{$m$-derived perfect} if $G^{(m)}$ is perfect, and \emph{derived perfect} if there is an $m$ such that it is $m$-derived perfect.
		\end{enumerate}
	\end{definition}
	
	Note that for a solvable group, the property of being $m$-step solvable is equivalent to being $m$-derived perfect, see also Fig.~\ref{fig:relation-solvable-weaklyperfect}.
	
	\begin{proposition}\label{lem:conditions_weaklyperfect}
		Let $m\geq 0$ be an integer, and let $G$ be a profinite group.
		Then the following conditions \cref{lem:conditions_weaklyperfect-0,lem:conditions_weaklyperfect-1,lem:conditions_weaklyperfect-2,lem:conditions_weaklyperfect-3} are equivalent:
		\begin{enumerate}
			\item\label{lem:conditions_weaklyperfect-0}
			The  $m$-derived subgroup $G^{(m)}<G$ is perfect.
			\item\label{lem:conditions_weaklyperfect-1}
			The derived series $\{G^{(k)}\}_{k \geq 0}$ of $G$ stabilizes at stage $m$, i.e., $G^{(m)} = G^{(k)}$ for all $k \geq m$.
			\item\label{lem:conditions_weaklyperfect-2}
			The maximal $(m+k)$-step solvable quotient $G^{m+k}$ is $m$-step solvable for all $k\geq 1$.
			\item\label{lem:conditions_weaklyperfect-3}
			The group $G$ is an extension of a perfect group by an $m$-step solvable group; that is, there exists a short exact sequence
			\begin{equation*}
				1 \to P \to G \to S \to 1,
			\end{equation*}
			where $P$ is a perfect group and $S$ is an $m$-step solvable group.
		\end{enumerate}
	\end{proposition}
	
	\begin{proof}
		The first equivalence \cref{lem:conditions_weaklyperfect-0}$\Leftrightarrow$\cref{lem:conditions_weaklyperfect-1} is immediate. The second one \cref{lem:conditions_weaklyperfect-1} $\Leftrightarrow$ \cref{lem:conditions_weaklyperfect-2} follows from the equality $(G/G^{(m+1)})^{(m)}=G^{(m)}/G^{(m+1)}$, as condition \cref{lem:conditions_weaklyperfect-2} is equivalent to $G^{(m)}=G^{(m+1)}$.
		
		\medskip
		
		We have the following short exact sequence:
		\begin{equation*}
			1 \to G^{(m)} \to G \to G^m \to 1
		\end{equation*}
		Since $G^m$ is $m$-step solvable by definition, assuming $G^{(m)}$ is perfect directly shows the implication \cref{lem:conditions_weaklyperfect-1}$\Rightarrow$\cref{lem:conditions_weaklyperfect-3}. Finally, we show the implication \cref{lem:conditions_weaklyperfect-3}$\Rightarrow$\cref{lem:conditions_weaklyperfect-0}. If $m = 0$, then $G \cong P$ is perfect, and \cref{lem:conditions_weaklyperfect-0} follows. We may thus assume $m > 0$. Since $G^m$ is the maximal $m$-step solvable quotient of $G$, there exists a surjection $G^m \twoheadrightarrow S$. This implies $G^{(m)} < P  < G$ and we have the following inclusions of groups
		\begin{equation*}
			P^{(m)} \lesssim G^{(m)} \lesssim P \lesssim G.
		\end{equation*}
		As $P$ is perfect, we have $P^{(m)}= P$. Therefore, $G^{(m)} = P$ is also perfect.
	\end{proof}
	
	\begin{figure}
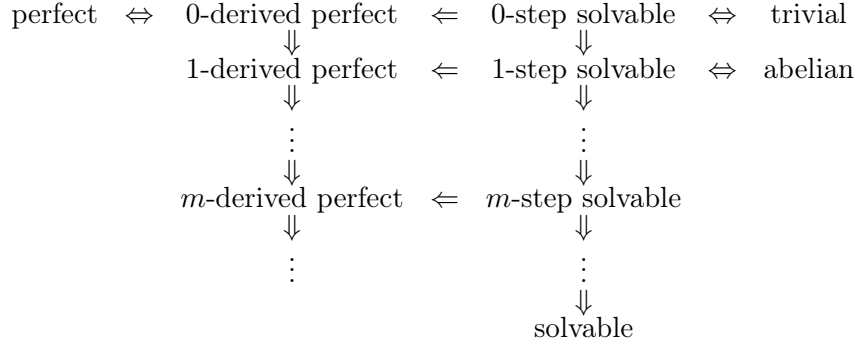

		\begin{equation*}\def\arraystretch{.7} 
			\begin{array}{ccccccc}
				\text{perfect} & \Leftrightarrow & 0\text{-derived perfect}
				& \Leftarrow      & 0\text{-step solvable}  & \Leftrightarrow & \text{trivial}                                            \\
				&                 & \Downarrow              &                 & \Downarrow             &                 &                \\
				&                 & 1\text{-derived perfect} & \Leftarrow      & 1\text{-step solvable} & \Leftrightarrow & \text{abelian} \\
				&                 & \Downarrow              &                 & \Downarrow             &                 &                \\
				&                 & \vdots                  &                 & \vdots                 &                 &                \\
				&                 & \Downarrow              &                 & \Downarrow             &                 &                \\
				&                 & m\text{-derived perfect} & \Leftarrow      & m\text{-step solvable} &                 &                \\
				&                 & \Downarrow              &                 & \Downarrow             &                 &                \\
				&                 & \vdots                  &                 & \vdots                 &                 &                \\
				&                 &                         &                 & \Downarrow             &                 &                \\
				&                 &                         &                 & \text{solvable}        &                 &
			\end{array}
		\end{equation*}
		\caption{Relations between step solvability and derived perfectness}\label{fig:relation-solvable-weaklyperfect}
	\end{figure}
	
	\begin{lemma}\label{lem:weaklyprefect_transitive}
		Let $m\geq 0$ be an integer. Consider the following exact sequence of profinite groups:
		\begin{equation*}
			1 \to H \to G \to Q \to 1
		\end{equation*}
		If $G$ is $m$-derived perfect and $Q$ is solvable, then $H$ is $m$-derived perfect.
	\end{lemma}
	
	\begin{proof}
		Let $k \in \mathbb{N}$ be the derived length of $Q$. Then $G^{(k)} \subset H \subset G$. Therefore, we obtain the following sequence of subgroups of $G$:
		\begin{equation*}
			G^{(k+m)} < H^{(m)} < G^{(m)}
		\end{equation*}
		Since $G$ is $m$-derived perfect, we have $G^{(m)} = G^{(k+m)}$, thus $H^{(m)} = G^{(m)}$, and it follows that $H^{(m)}$ is perfect. Thus $H$ is $m$-derived perfect.
	\end{proof}
	
	\begin{proposition}\label{weaklyperfect-characterization}\label{cor:infinitederived}
		A non-perfect hyperbolic profinite $F$-group is not derived perfect and admits solvable quotients of derived length $k$ for any $k\geq 0$.
	\end{proposition}
	
	\begin{proof}
		Let us first prove the non-$m$-derived perfectness for any $m\geq 0$. By the non-perfectness hypothesis, we may assume $m > 0$. We first claim that a non-abelian torsion-free profinite $F$-group is not $m$-derived perfect for any $m$.
		Indeed, such a group admits a surjection onto $\widehat{\FF}_2$, the free profinite group of rank $2$. The claim follows from the fact that $\widehat{\FF}_2$ has quotients isomorphic to solvable $2$-generated groups with arbitrarily large derived length (for example, iterated wreath products $(\mathbb{Z}/2\mathbb{Z})\wr(\mathbb{Z}/2\mathbb{Z})\wr\cdots$ of length $d$).
		
		By Theorem~\ref{thm:proFFenchels}, there exists a torsion-free open normal subgroup $H$ of $\Delta$ such that $\Delta/H$ is solvable, and moreover, $H$ is non-abelian by the hyperbolicity. In the case where $\Delta$ is $m$-derived perfect, Lemma~\ref{lem:weaklyprefect_transitive} implies that $H$ is also $m$-derived perfect, which  contradicts the claim.
		
		Since the derived series  $\{\Delta^{(i)}\}_{i \geq 0}$ does not stabilize by Proposition~\ref{lem:conditions_weaklyperfect}~\ref{lem:conditions_weaklyperfect-2}, the second claim follows by forming the quotient by $\Delta^{(i)}$ for $i>>0$ .
	\end{proof}

	\subsubsection{} The smallest finite groups of derived length $3$ are the symmetric group $S_{4}$ and the special linear group $\mathrm{SL}_{2}(\mathbb{F}_{3})$, both of which have order $24$. We thus immediately obtain the following:
	\begin{equation*}
		\textit{If $\Delta$ is hyperbolic, then } \Card(\Delta^{3}) \geq 24.
	\end{equation*}
	By Table~\ref{fig:Table_nonhyperbolic_F_groups}, a profinite $F$-group with signature $(0,0;\{2,3,4\})$ is isomorphic to $S_{4}$.
	The next proposition shows that, conversely, this is the only case where $\Delta^{3}$ is isomorphic to $S_{4}$.
	
	\begin{proposition}\label{cor:s4}
		Let $\Delta$ be a profinite $F$-group.
		If $\Delta^{3}$ is isomorphic to $S_{4}$, then the signature of $\Delta$ is $(0,0;\{2,3,4\})$.
	\end{proposition}
	
	\begin{proof}
		Assume that $\Delta^{3}$ is isomorphic to $S_{4}$. By Theorem~\ref{thm:proFFenchels}~\ref{thm:proFFenchels-3}, we may assume that the signature of $\Delta$ is $(g,r;\{n_1,\dots,n_k\})$ with $(g,r)=(0,0)$, since otherwise $\Card(\Delta^{3})$ would be infinite. The derived series of $S_{4}$ is $\{1\}< V_{4}< A_{4}< S_{4}$, and the successive quotients are
		\begin{equation*}
			V_{4}\cong \mathbb{Z}/2\mathbb{Z}\times \mathbb{Z}/2\mathbb{Z},\quad
			A_{4}/V_{4}\cong \mathbb{Z}/3\mathbb{Z},\quad
			S_{4}/A_{4}\cong \mathbb{Z}/2\mathbb{Z}.
		\end{equation*}
		
		\medskip
		
		Let us compute the signatures of $\Delta^{(1)}$, $\Delta^{(2)}$, and $\Delta^{(3)}$. By $\Delta^{\mathrm{ab}}\cong\mathbb{Z}/2\mathbb{Z}$ and Lemma~\ref{lem:inducedsignature}, we have
		\begin{equation*}
			2=\Card(\Delta^{\mathrm{ab}})
			= \frac{\prod_{i=1}^k n_i}{\mathrm{lcm}(n_1,\dots,n_k)}.
		\end{equation*}
		Hence, after reordering the periods if necessary, we may write
		\begin{equation*}
			n_1 = 2a_1,\qquad
			n_2 = 2^{\alpha}a_2,\qquad
			n_i = a_i\ (i\ge 3),
		\end{equation*}
		where $\alpha\ge 1$, each $a_i$ is odd, and the $a_i$ are pairwise coprime. In this case, the induced signature of $\Delta^{(1)}$ is
		\begin{equation*}
			\bigl(0,0;
			\{a_{1},\ 2^{\alpha-1}a_{2},\
			a_3,a_3,\
			\dots,\
			a_{k},a_{k}\}\bigr).
		\end{equation*}
		Again, by $(\Delta^{(1)})^{\mathrm{ab}}\cong \mathbb{Z}/3\mathbb{Z}$ and Lemma~\ref{lem:inducedsignature}, we obtain
		\begin{equation*}
			3=\Card(\Delta^{(1)})^{\mathrm{ab}}
			=\frac{a_{1}\cdot 2^{\alpha-1}a_{2}\cdot a_3^2\dots a_{k}^2}
			{a_{1}\cdot 2^{\alpha-1}a_{2}\cdot a_{3}\dots a_{k}}
			=\prod_{i=3}^{k} a_{i},
		\end{equation*}
		which implies that $a_{3}=3$ and $a_{4}=\dots=a_{k}=1$. In this case, the induced signature of  $\Delta^{(2)}$ is
		\begin{equation*}
			\bigl(0,0;
			\{a_{1},a_{1},a_{1},\ 2^{\alpha-1}a_{2},2^{\alpha-1}a_{2},2^{\alpha-1}a_{2}\}\bigr).
		\end{equation*}
		Again, by $\Delta^{(2)}/\Delta^{(3)}\cong \mathbb{Z}/2\mathbb{Z}\times \mathbb{Z}/2\mathbb{Z}$ and Lemma~\ref{lem:inducedsignature}, we obtain
		\begin{equation*}
			4=\Card(\Delta^{(2)})^{\mathrm{ab}}
			=\frac{a_{1}^{3}\,2^{3(\alpha-1)}a_{2}^{3}}
			{a_{1}\,2^{\alpha-1}a_2}
			= a_{1}^{2}\,2^{2(\alpha-1)}\,a_{2}^{2}.
		\end{equation*}
		Since $a_{1}$ and $a_{2}$ are odd, this forces $a_{1}=a_{2}=1$ and $\alpha=2$.
		Therefore the signature of $\Delta$ is $(0,0;\{2,3,4\})$ as desired.
	\end{proof}
	
	By similar computations, we can show that $\operatorname{SL}_2(\FF_3)$ group does not occur as $\Delta^{3}$ for some profinite $F$-group $\Delta$.
	
	\subsubsection{}\label{subsub:anabStepGrp} 
	We already established that for a profinite $F$-group, the hyperbolicity and affineness properties are group-theoretic, see Proposition~\ref{prop:SigIsoInv}. Let us consider the characterization by $m$-step solvable quotient for $m \geq 3$. We recall that the Chen ranks $\{\Theta_k(\Delta)={\rm dim}({\rm gr}_k\Delta^2)\}_{k\geq 1}$, of a group are also some group-isomorphism invariants -- see \cite{Chen51}.

	\begin{proposition}\label{prop:3steprecoofaffine}
		Let $\Delta$ be a non-perfect profinite $F$-group. 
		\begin{itemize}
			\item[$(i)$] The group $\Delta$ is hyperbolic if and only if $\Delta^3$ is of derived length $3$ and not isomorphic to~$S_4$. 
			\item[$(ii)$] Assume that $\Delta$ is hyperbolic. Then $\Delta$ is affine if and only if $\Delta^3$ contains an open normal subgroup $H$ such that $\Delta^3/H$ is abelian, $H^2$ is torsion free, and such that
			\[
			2\Theta_2(H)-\Theta_1(H)^2+\Theta_1(H)=0.
			\]
		\end{itemize} 
	\end{proposition}
	
	In particular, if two non-perfect profinite $F$-groups have isomorphic $3$-step solvable quotients, they are both affine or hyperbolic as soon as one is so.
	
	\medskip
	
	Note that the hyperbolicity property cannot be characterized using the $2$-step quotient only, since curves of signature $\Sigma_n=(0,0;\{2,2^n,3\})$ all have $S_3$ as $2$-quotient, while the curve corresponding to $\Sigma_1$ is non-hyperbolic. For hyperbolic curves to have a $3$-step quotient of derived length $3$ stands in sharp contrast to the non-hyperbolic case, since the derived length of every non-perfect, non-hyperbolic profinite $F$-group is at most $2$ -- except for the case $(0,0;\{2,3,4\})$, see Table.~\ref{fig:Table_nonhyperbolic_F_groups}.
	
	\medskip

	The proof of Proposition~\ref{prop:3steprecoofaffine} relies on the following lemma.
	
	\begin{lemma}\label{lemma:torsion}
		Let $\Delta$ be a non-perfect profinite $F$-group. Then $\Delta$ is torsion-free if and only if $\Delta^2$ is torsion-free.
	\end{lemma}
	
	\begin{proof}
		If $\Delta$ is torsion-free, then it is a free profinite group or a profinite surface group.
		In both cases $\Delta^2$ is also torsion-free.
		
		\medskip
		
		Let us assume that $\Delta^2$ is torsion-free and fix $i\in\{1,\dots, k\}$ -- assuming $\Delta$ has at least one torsion element. If $r\geq 1$, then the image of $\delta_{i}$ in $\Delta^{\mathrm{ab}}$ is a non-trivial torsion element, which is a contradiction. We may thus further assume $r=0$. 
		
		\medskip
		
		Let us start by dealing with the $g=0$ case. The assumption on $\Delta^2$ gives that the image of $\delta_i$ in $\Delta^{\mathrm{ab}}$ is trivial. By Lemma~\ref{lem:orderinab}, the order of the image of $\delta_i$ in $\Delta^{\mathrm{ab}}$ is $\gcd\bigl(n_i, \mathrm{lcm}(n_j\mid j\neq i)\bigr)=1$, hence $n_{i}$ is coprime to every period $n_{j}$ for $j\neq i$. However, as $\Delta$ is non-perfect this is a contradiction with Proposition~\ref{prop:charperfect}.
		
		\medskip
		
		It remains to consider the case $g\geq 1$. Let $\ell$ be a prime number dividing $n_i$ and consider the Heisenberg group
		\[
		H_{\ell}=
		\langle x,y,z \mid x^{\ell}=y^{\ell}=z^{\ell}=1,
		[x,y]=z^{-1},
		[x,z]=[y,z]=1\rangle.
		\]
		It is a finite metabelian group. The map defined by
		\[
		\alpha_1\mapsto x,~\beta_1\mapsto y, ~ \delta_i\mapsto z
		\]
		and by sending all other generators of $\Delta$ to $1$ gives a homomorphism $\psi_i\colon\Delta\to H_{\ell}$. This homomorphism factors through $\Delta^2$ while the image of $\delta_i$ has order $\ell$. This is a contradiction with our starting assumption, hence $\Delta$ is torsion free.  
	\end{proof}

	\begin{proof}[Proof of Proposition \ref{prop:3steprecoofaffine}]
		We first treat the characterization of hyperbolicity.  If $\Delta$ is non-hyperbolic, then we conclude by the classification of Table~\ref{fig:Table_nonhyperbolic_F_groups} where all groups have derived length at most $2$ except $S_4$. For the converse, by Proposition~\ref{cor:infinitederived}, there is a solvable quotient $Q$ of $\Delta$ of derived length at least $4$. Then, $Q/Q^{(3)}$ is solvable of derived length $3$. As it is a subquotient of $\Delta^3$ we get that $\Delta^3$ is not $2$-step solvable. 
		
		\medskip
		
		Let us consider the affineness property. We can apply the affineness criterion of Theorem~\ref{thm:proFFenchels}~\cref{thm:proFFenchels--1}. That is, from the existence or non-existence of a torsion-free open normal subgroup whose quotient is abelian, we can distinguish via $\Delta$ the following two cases:
		\begin{center}
			(i) $r \geq 1$, or $r=0$ and $\Delta$ satisfies the abelian-period condition. \\
			(ii) $r=0$ and $\Delta$ does not satisfy the abelian period condition. 
		\end{center}
		
		Let us first show that we can, in fact, distinguish between these two situations from $\Delta^3$ alone. In the first case, the torsion-free open normal subgroup $\tilde{H}$ of abelian quotient given by Theorem~\ref{thm:proFFenchels}~\ref{thm:proFFenchels--1} provides an image $H$ in $\Delta^{3}$, such that $\tilde{H}^{2}\rightarrow H^{2}$ is an isomorphism by Lemma~\ref{lemma:mstep}. In particular, by Lemma~\ref{lemma:torsion} the subgroup $H^2$ is torsion-free if and only if $\tilde{H}$ is torsion-free. Thus we are in case (i) if and only if $\Delta^3$ admits an open normal subgroup $H$ such that $\Delta^3/H$ is abelian and $H^2$ is torsion-free..    
		
		\medskip
		
		Furthermore, in the situation (i), since $\tilde{H}$ is torsion-free, the group $\tilde{H}^{2}$ is the maximal metabelian quotient of either a non-abelian free profinite group $F$ or of a non-abelian profinite surface group $S$. In both cases, we have
		\[
		\Theta_1(H)=\operatorname{rank}(H^2)^{\mathrm{ab}}\text{ and } 2\Theta_2(H)=
		\begin{cases}
			(\operatorname{rank}(H^2)^{\mathrm{ab}})^2-\operatorname{rank}(H^2)^{\mathrm{ab}}\ \text{ for } F, \text{ affine case}\\
			(\operatorname{rank}(H^2)^{\mathrm{ab}})^2-\operatorname{rank}(H^2)^{\mathrm{ab}}-2\text{ for } S, \text{ proper case},
		\end{cases}
		\]
		see \cite{Chen51} and \cite{Mur66} \S~3 for explicit formulae (resp. \cite{MR3926216} Eq.~(45)) for $F$ (resp. for $S$), and we obtain the given numerical affineness characterization.
	\end{proof}
	
	In the non-hyperbolic case we can also recover affineness by the $3$-step quotient from a direct computation using Table~\ref{fig:Table_nonhyperbolic_F_groups}.
	
	\section{Anabelian properties of Deligne-Mumford curves}\label{sec:anabprop}
	In what follows, we assume an algebraic Deligne-Mumford stack $\XX$ to be irreducible separated and smooth over a field $K$ that is of characteristic zero. For arithmetic and homotopic purposes as in Section~\ref{subsubs:FondES}, we further assume that $\XX$ is quasi-compact and geometrically connected. Let us recall that, to a geometric point $\bar{x}\colon \spec \bar{K}\to \XX$, we can attach the (finite) stack inertia group of its $2$-transformation $\II_{\XX,\bar{x}}=\spec \bar{K}\prescript{}{\bar{x}}{\times} \II$, where $\II=\XX\times_{\XX\times\XX}\XX$ denotes the inertia stack of $\XX$. These groups will be of central interest in our anabelian results.
	
	\medskip
	
	We start by introducing the necessary basics on Deligne-Mumford curves towards the rigidification and coarsification homotopy sequences in Section~\ref{subsubs:HoRigCoa}. We obtain the group-theoretic characterization of some discrete invariants such as the genus $g$ and number of cusps $r$, in the $m$-step situation for those, and in particular the stack inertia $\II_*$ of the generic and closed points as subgroups of the étale fundamental group. This is done in Sections~\ref{subsec2.1.2} and \ref{subsub:rigidification} which also contains some first anabelian results for $\XXrig$ and $\XXcoa$. A homotopic and algebraic uniformization result for Deligne-Mumford curves as $\II_{\XX,\eta}$-gerbe over their rigidification is also established, see Proposition~\ref{prop:monoXoverXrig}. This finally leads to our main anabelian result for Deligne-Mumford curves -- see Section~\ref{subsub:DManab}.
	
	\medskip
	
	Building on the group-theoretic profinite $F$-group constructions of Sections~\ref{sub:Fenchel_ref} and \ref{sec:weaklyperfect-hyperbolic} -- which rely on cuspidal and stack inertial properties -- we obtain in Section~\ref{subsection:2.2} our main $5$-step anabelian result for stacky curves (or the rigidification $\XXrig$ of a Deligne-Mumford curve). Even though the generalization of the $m$-step question fails directly in the context of Deligne-Mumford curves, we end by providing an anabelian result in the particular case of a trivial gerbe over an affine hyperbolic curve.

	\subsection{Arithmetic homotopy of Deligne-Mumford curves}\label{sub:AHGDM}
	
	\subsubsection{} 
	A Deligne-Mumford curve $\XX$ is a dimension-one Deligne-Mumford stack. We denote by $g_\XX=\dim H^1(\XX,\mathcal O_\XX)$ its genus, by $\{n_1,\dots,n_k\}$ the list of its periods -- that is, the order of the finite stack inertia group $\mathcal{I}_{\XX,x_i}$ -- and by $\II_{\XX,\eta}$ its generic stack inertia group.
	
	\begin{definition}
		A  smooth geometrically connected Deligne-Mumford curve $\XX$ over a field $K$ of characteristic $0$ is \emph{hyperbolic} (resp. affine) if there exists a compactification $\XX\hookrightarrow \tilde{\XX}$ by a smooth proper curve $\tilde{\XX}$ over $K$ such that $\mathcal{D}=\tilde{\XX}\setminus \XX$ is a Cartier divisor given as a finite sum of $r_\XX$ points and such that its Euler characteristic
		\begin{equation}\label{eq:EulerDM}
			\chi(\XX)\coloneqq\frac{1}{\Card(\II_{\XX,\eta})}\ \left(2 - 2g_\XX - r_\XX - \sum_{i=1}^{k} (1 - \frac{1}{n_i})\right)
		\end{equation}
		is negative (resp. if $r_{\mathcal{X}}\geq 1$). A \emph{stacky curve} is a Deligne-Mumford curve with a dense open subscheme; it is hyperbolic or affine under similar corresponding assumptions. The closed points of $\mathcal{D}$ are called \emph{cuspidal points}.
	\end{definition}
	This definition of stacky curves is equivalent to the definition of ``orbicurves'' of \cite{MR2046610}, whose terminology is more analytic-orbifold evoking. 
	
	\begin{proposition}
		A Deligne-Mumford curve $(\tilde{\XX},\mathcal{D})$ over $K$ admits
		\begin{enumerate}
			\item \emph{A rigidified algebraic Deligne-Mumford stack $\XXrig$}, defined as the $G_\eta$-gerbe $\XX\to\XXrig=\XX\myfatslash G_\eta$ via the $2$-quotient $\XX\myfatslash G_\eta$, where $G_\eta$ is the Zariski closure of the generic inertia $\II_{\XX,\eta}<\II_\XX$ (similarly for the divisor $\mathcal{D}\to\mathcal{D}_{\rm rig}$).
			\item \emph{A coarse moduli space $\XXcoa$ }, that is a smooth curve over $K$ whose geometric points are in bijection with the geometric fiber of $\XX$.
		\end{enumerate}
	\end{proposition}
	In the rigidification, the $2$-quotient removes $\II_{\XX,\eta}$ from the inertia stack $\II_\XX$ of $\XX$ as well as of $\mathcal{D}$; we refer to \cite{AOV08} Appendix~A for its construction. The coarse moduli is endowed with a morphism $\pi\colon\XX \to \XXcoa$ that is universal among the morphisms from $\XX$ to $K$-schemes. For Deligne-Mumford stacks, the existence of $\XXcoa$ is given by \cite{KM07} Theorem~1 since $\XX$ has finite diagonal. In the case of Deligne-Mumford curves, one can give a direct construction that follows from the étale local description of $\XX$ as a quotient stack as follows.

	Since $\XX$ and its rigidification $\XXrig$ have the same coarse moduli, one can assume that $\XX$ has no generic inertia. It is thus generically covered by a scheme $X^0$ and by a finite number of quotients $[\spec A/G_i]$, where $A$ is a $1$-dimensional $K$-algebra and $G_i$ the finite stabilizing group of a stacky point of $\XX$. The coarse moduli is obtained as the gluing of $X^0$ and of $U_i=\spec A^{G_i}$, which can be done since a curve is determined by its fraction field and its ramification data. One then checks directly that the resulting map $\XX \to \XXcoa$ satisfies the universal coarse moduli property. Since $\XXcoa$ has at worst quotient singularities, it is thus a normal curve, hence smooth.

	\medskip
	
	The moduli spaces of curves $\mathcal{M}_{g,[r]}$, whose $S$-fibers are families of curves $C\to S$ of genus $g$ with Cartier divisor of degree $r$, provide examples of hyperbolic Deligne-Mumford curves for $(g,m)=(0,4)$ and $(1,1)$. Their generic inertia groups are, respectively, the Vierergruppe $V_4$ and the involution group $\ZZ/2\ZZ$.
	
	\begin{remark}
		By definition of $\mathcal{D}$, and following \cite{AOV08} \S~3 for the genus equality, one has $(g_\XX,r_\XX)=(g_{\XXcoa}, r_{\XXcoa})$.
	\end{remark}
	
	\subsubsection{}\label{subsub:InertiaHomotopy} 
	Fixing a geometric point $x$ of $\XX$, consider the étale fundamental group $\Pi_{\XX}\coloneqq \pi_1^{\mathrm{\acute et}}(\mathcal X,x)$ and the geometric fundamental group $\Delta_\XX \coloneqq \pi_1^{\mathrm{\acute et}}(\mathcal X\times {\overline K})$ -- we refer to \cite{NOO04} for a definition in terms of Galois category and to \cite{Zoo01} in terms of étale topological type; since $\XX$ is normal, they are both equivalent by \cite{AM69} Theorem~11.1. 
	
	\medskip
	
	Fixing $k\hookrightarrow \CC$ and denoting $\XX(\CC)^{\rm an}$ for the analytic orbifold associated to $\XX$, it follows from a Riemann existence theorem, as in \cite{FRI82} Theorem~8.4, that we can apply \cite{MR2279100} Proposition~5.6.
	\begin{proposition}
		The geometric fundamental group $\Delta_\XX$ of a stacky curve $\XX$ is a profinite $F$-group that identifies as $\Delta_\XX\simeq \widehat{\pi}_1^{\rm orb}(\XX(\CC)^{\rm an},x)$. If $\XX$ is of genus $g$ with $r$ cuspidal points and $k$ orbifold points of periods $\{n_{\mathcal{X},1},\dots,n_{\mathcal{X},k_{\mathcal{X}}}\}$, then the signature of $\XX$ is given by $\Sigma_\XX=(g_{\mathcal{X}}, r_{\mathcal{X}}; \{n_{\mathcal{X},1},\dots,n_{\mathcal{X},k_{\mathcal{X}}}\})$.
	\end{proposition}

	In particular, the Euler characteristic of a stacky curve is given as in Formula~\ref{def:euler-characteristic}, and we remark that the Euler characteristic Eq.~\eqref{eq:EulerDM} is just $\chi(\XX)=\chi(\XXrig)/\Card(\II_{\XX,\eta})$, so 
	\begin{center}
		\emph{$\XX$ is hyperbolic if and only if $\XXrig$ is hyperbolic}.
	\end{center} 
	
	Similarly, $\XX$ is affine if and only if $\XXrig$ is affine.  In a similar spirit, we say that a Deligne-Mumford curve is \emph{non-perfect} if its geometric fundamental group $\Delta_{\mathcal{X}}$ is non-perfect.

	Note that by Proposition~\ref{prop:HoRigSeq}, the geometric fundamental group $\Delta_\XX$ of a Deligne-Mumford curve is the extension of a finite group $\II_{\XX,\eta}$ by a profinite F-group $\piet(\XXrig\times \bar{K})$ with presentation as above.
	
	\medskip
	
	We further recall that to each closed point $x\in \mathcal{D}$ of the boundary is associated, up to conjugacy, a \emph{cuspidal inertia group} $\II_{x,\mathcal{D}}\hookrightarrow \Pi_\XX$, and that an étale local description gives $\II_{x,\mathcal{D}}\simeq \widehat{\ZZ}(1)\rtimes H$, where $H$ is a finite stack inertia group.

	Similarly, stack inertia groups are \emph{homotopic subgroups}, since by \cite{NOO04} Section~4, one has by Galois category formalism, for every geometric point $\bar{x}$ of $\XX$, a group homomorphism $\omega_{\bar{x}}\colon\II_{\XX,\bar{x}}\to \Pi_\XX$. It then follows from the uniformizability of a Deligne-Mumford curve of Proposition~\ref{prop:monoXoverXrig} -- that is $\XX\simeq [U/H]$ with $U$ algebraic space and $H$ finite group -- and Theorem~6.2 ibid. that these homomorphisms $\omega_x\colon\II_{\XX,\bar{x}}\hookrightarrow \Pi_\XX$ are injective.
	
	\begin{remark}
		The uniformizability of the Deligne-Mumford curves can also be deduced, by Theorem~6.2 ibid., from the injectivity $\II_{\XXrig,\bar{x}}\hookrightarrow \pi_1^{\rm orb} (\XXrig\times \bar{K})$, the residual finiteness of F-groups by \cite{Sah69} Theorem~1.5~(a), and the fact that $\II_{\XXrig,\bar{x}}\simeq \II_{\XX,\bar{x}}/\II_{\XX,\eta}$ as in Proposition~\ref{prop:HoRigSeq}.
	\end{remark}

	\subsubsection{}\label{subsubs:HoRigCoa} 
	At the homotopy level, the rigidification map $\XX\to\XXrig$ and the coarse map $\XX\to\XXcoa$ give rise to two short exact sequences. The following is an algebraic version of \cite{MR2279100} Proposition~7.5 that holds for the geometric group of hyperbolic analytic Deligne-Mumford curves.
	\begin{proposition}\label{prop:HoRigSeq}
		A  Deligne-Mumford curve over $K$ admits a homotopy-rigidification exact sequence
		\begin{equation}\label{Eq:inRig}
			1 \longrightarrow \II_{\XX,\eta} \longrightarrow \pi_1^{\rm et}(\XX,\bar{x}) \longrightarrow \pi_1^{\rm et}(\XXrig,\bar{x}) \longrightarrow 1,
		\end{equation}
		and similarly for the geometric fundamental groups. 
	\end{proposition}
	
	\begin{proof}
		Since a Deligne-Mumford curve is geometrically unibranch, the exact sequence is given by \cite{LV18} Corollary~2.9, where it is sufficient to check the left exactness at the analytification level, since $\widehat{\pi}_1^{\rm et}(\XX(\CC)^{\rm an})\simeq\Delta_\XX\hookrightarrow\Pi_\XX$:
		Since $\XX(\CC)^{\rm an}$ is one-dimensional, an application of van Kampen theorem as in \cite{NOO04} Example~6.3 implies the injectivity $\II_{\XX,\eta}\hookrightarrow \Delta_\XX$.
		The exact sequence for the geometric fundamental groups is given by \cite{MR2279100} Proposition~7.5.
	\end{proof}
	
	
	The following is \cite{NOO04} Theorem~7.12. 
	\begin{proposition}\label{prop:hocoar}
		Let $\XX$ be a Deligne-Mumford curve, and denote $N_\XX=\langle \II_{\XX,x}\rangle$ the closure in $\pi_1^{\rm et}(\XX,x)$ of the subgroup generated by the inertia groups of all the closed geometric points of $\XX$.  One has a homotopy-coarsification exact sequence
		\begin{equation}\label{Eq:inCoa}
			1 \longrightarrow N_\XX \longrightarrow \pi_1(\XX,x) \longrightarrow \pi_1(\XXcoa,x) \longrightarrow 1.
		\end{equation}
	\end{proposition}
	Note that, since a hidden étale path $x\rightsquigarrow x'$ between closed points has the effect of conjugating the stack inertia groups, the group $N_\XX$ is a normal subgroup of $\pi_1(\XX,x)$.
	
	\subsubsection{}\label{subsubs:FondES}
	For $\XX$ a Deligne-Mumford stack over $K$ that is quasi-compact and geometrically connected, it follows from \cite{Zoo01} Cor.~6.6 that one has a fundamental arithmetic-geometric exact sequence 
	\[
	1\to\Delta_\XX\to\Pi_\XX\to \Gk\to 1,\text{ which gives an outer Galois representation } \Gk\to \Out(\Delta_\XX)
	\]
	where $\mathrm{Out}(\Delta_{\mathcal{X}})$ denotes the outer automorphism group of $\Delta_{\mathcal{X}}$ -- the quasi-compactness implies the right-exactness, and the geometric connectedness implies the left-exactness.
	
	For simplicity, let us assume that a splitting of the exact sequence above is given, for example by a rational point $s\colon \spec K\to \XX$, so that 
	\begin{center}\itshape 
		the previous outer Galois representation lifts to a Galois action $\varphi_s\colon G_K\longrightarrow \Aut(\Delta_\XX)$.
	\end{center}
	\medskip
	
	As a substitute for a general description of this action -- including the case where $\XX$ is a global quotient as in Proposition~\ref{prop:monoXoverXrig} so that the cuspidal inertia is potentially of the form $I_D\simeq \widehat{\ZZ}\rtimes H$ with $H$ \emph{non-cyclic} finite group -- we recall the motivational example of $\XX=\MM_{g,[r]}$ the moduli stack of curves of genus $g$ with $r$ marked points. They are Deligne-Mumford over $\QQ$, and the above $G_\QQ$-action $\varphi_s$ is given on an inertia group $\II \leq \Pi_\XX$ by conjugacy-cyclotomy, that is:
	\[
	\sigma.\gamma=g_{\gamma,\sigma}^{-1}.\gamma^{\chi_\sigma}.g_{\gamma,\sigma} \text{ where } \gamma\in \II \text{ and for }\sigma\in G_\mathbb{Q}
	\]
	in the situations where
	\begin{enumerate}
		\item $\II\simeq \II_{D_i}\leq \Pi_\XX$ is a \emph{cuspidal inertia group} associated to an irreducible component $D_i$ of the border of the Deligne-Knudsen compactification of $\MM_{g,[r]}$, see for example \cite{Nak02};
		\item  $\II\simeq\II_{\XX,\bar{x}}\leq \Pi_\XX$ is a \emph{cyclic} stack inertia group -- see \cite{CM23}.
	\end{enumerate}
	which both rely on the existence of a compactification of $\XX$ with normal crossing divisors and an intricate process in tracking the geometry of curves.

	\medskip
	
	In the following, we further consider the (geometrically) \emph{maximal $m$-step solvable quotient\footnote{The anabelian literature commonly uses the notation $\Pi_{\mathcal{X}}^{(m)}$, which has been updated so that it better reflects the non-canonical choice of taking the geometric quotient and that it respects the group theory notation for derived groups.} $\Pi_\XX^{\Delta-m}\coloneqq \Pi_{\mathcal{X}}/\Delta^{(m)}_{\mathcal{X}}$}, which also fits into a natural exact sequence
	\begin{equation}\label{def:hom-exact-2}
		1 \to \Delta_{\mathcal{X}}^{m} \to \Pi_\XX^{\Delta-m}\to G_K\to 1, \text{ and gives } \Gk\to \Out(\Delta_{\mathcal{X}}^{m})
	\end{equation}
	as before. In particular, this endows $\Delta_{\mathcal{X}}^{\mathrm{ab}}$ with a structure of $\Gk$-module.

	\subsection{The Grothendieck conjecture for Deligne-Mumford curves}\label{sub:GCDM}
	
	\subsubsection{}\label{subsec2.1.2} 
	The group theoretic characterization of the genus $g_\XX$ (resp. the number of cusps $r_\XX$) follows from group-theoretic data at the level of the coarse moduli space $\XXcoa$ which is a smooth (schematic) curve (resp. at the level of the rigidification $\XXrig$, which is a smooth stacky curve).
	
	\medskip
	
	Indeed, since by \cite{NOO04} Theorem~7.12, the kernel of $\Delta_\XX^{\rm ab}\to \Delta_{\XXcoa}^{\rm ab}$ coincides with the subgroup of $\Delta_\XX^{\rm ab}$ generated by all the torsion subgroups, one obtains an isomorphism of $G_K$-modules $\Delta_{\XXcoa}^{\rm ab}\simeq \Delta_{\XX,{\rm tf}}^{\rm ab}$. One can then proceed in terms of Frobenius weight as in \cite{MR1072981} \S~2.4 by considering the following exact sequence (resp. isomorphism) of $G_K$-modules:
	\begin{equation}\label{wf2}
		0 \rightarrow \hat{\mathbb{Z}}(1) \rightarrow \mathbb{Z}[D(\overline{K})]\bigotimes_{\mathbb{Z}} \hat{\mathbb{Z}}(1)\xrightarrow{f} \Delta_{\XXcoa}^{\mathrm{ab}}\rightarrow T(J_{\XXcoa})
		\rightarrow 0 \quad \text{when } r_{\mathcal{X}}\neq 0
	\end{equation}
	(resp. $\Delta_{\XXcoa}^{\mathrm{ab}}\xrightarrow{\sim}T(J_{\XXcoa})$, when $r_{\mathcal{X}}=0$), where $\mathbb{Z}[D(\overline{K})]$ is the free $\mathbb{Z}$-module with basis the cusps $D(\overline{K})$ of $\XXcoa$, $\hat{\mathbb{Z}}(1)$ denotes the Tate twist, and $f$ sends $v\otimes 1$ to a (topological) generator of the (cuspidal) inertia group $I_{v,\Delta_{\XXcoa}^{\mathrm{ab}}}$ at $v\in D(\overline{K})$. The $G_{K}$-representations on $\mathbb{Z}[D(\overline{K})]\bigotimes_{\mathbb{Z}} \hat{\mathbb{Z}}(1)$ and $T(J_{\XXcoa})$ then have Frobenius weights $-2$ and $-1$, respectively, by ibid. -- see also \cite{MR4745885} Lemma~1.2 for a positive characteristic version.
	
	\begin{proposition}\label{lem:reconstructionsgr}
		Let $\mathcal{X}$ and $\mathcal{X}'$ be Deligne-Mumford curves over a field $K$ that is finitely generated over $\mathbb{Q}$, and assume that $\XXrig$ is non-perfect. In this case, if $\Pi_{\mathcal{X}}^{\Delta-m}$ and $\Pi_{\mathcal{X}'}^{\Delta-m}$ are isomorphic over $G_K$ for $m=1$ (resp. for $m=3$), then $g_{\mathcal{X}}=g_{\mathcal{X}'}$  (resp. $r_{\mathcal{X}}=r_{\mathcal{X}'}$). 
	\end{proposition}
	
	\begin{proof}
		We proceed as described above under the identification $\Delta_{\XXcoa}^{\rm ab}\simeq \Delta_{\XX,{\rm tf}}^{\rm ab}$. Let  $W_2$ be the maximal $G_K$-submodule of $\Delta_{\XX,{\rm tf}}^{\rm ab}$ that has weight $-2$.
		Then we have
		\begin{align}
			&\mathrm{rank}_{\hat{\mathbb{Z}}}(W_2)
			= r_{\mathcal{X}}-\varepsilon_{\mathcal{X}} \label{eq:cups}\\ 
			&\mathrm{rank}_{\hat{\mathbb{Z}}}(\Delta_{\XX,{\rm tf}}^{\rm ab}/W_2)
			= 2g_{\mathcal{X}}
		\end{align}
		where $\varepsilon_{\mathcal{X}}=0$ (resp. $\varepsilon_{\mathcal{X}}=1$) if $r_{\mathcal{X}}=0$ (resp. $r_{\mathcal{X}}\ge 1$).
		If $\Pi_{\mathcal{X}}^{\Delta-m}$ and $\Pi_{\mathcal{X}'}^{\Delta-m}$ are isomorphic over $G_{K}$ for $m\geq 1$,
		then $\Delta_{\XX,{\rm tf}}^{\rm ab}$ and $\Delta_{\mathcal{X}',{\rm tf}}^{\mathrm{ab}}$ are isomorphic as $G_{K}$-modules, and it follows directly that $g_{\mathcal{X}}=g_{\mathcal{X}'}$.
		
		\medskip
		
		The affineness characterization by the $3$-step quotient of $\XXrig$ -- which, group-theoretically, is obtained by Proposition~\ref{prop:IrigcoaRec} -- as in Proposition~\ref{prop:3steprecoofaffine}~(ii) implies $r_{\mathcal{X}}\geq 1$ if and only if $r_{\mathcal{X}'}\geq 1$, then Eq.~\eqref{eq:cups} gives $r_{\mathcal{X}}=r_{\mathcal{X}'}$.
	\end{proof}
	
	\begin{remark}
		As in the affineness characterization of Proposition~\ref{prop:3steprecoofaffine}~(ii) one can give a mono-anabelian characterization of $g_\XX$ and $r_\XX$, for example
		$g_\XX=1/2\,\mathrm{rank}_{\hat{\mathbb{Z}}}(\Delta_{\mathcal{X}}^{\mathrm{ab}/{\rm tor}}/W_{2})$.
		The characterization of $r_\XX$ is more intricate and involves a certain group $H$ and the equation on Chen ranks, see ibid.    
	\end{remark}
	
	\begin{wrapfigure}[5]{r}{4cm}
		\vspace*{-2em}
		\centering
		\begin{tikzcd}[row sep=12pt, column sep=5pt]
			\Delta_\XX\ar[d,two heads] \ar[r]& \Pi_\XX \ar[d, two heads]\ar[dr]& \\
			\Delta_{\XXrig} \ar[d,two heads] \ar[r]& \Pi_{\XXrig} \ar[d, two heads]\ar[r]& G_K \\
			\Delta_{\XXcoa}\ar[r] & \Pi_{\XXcoa} \ar[ur] & 
		\end{tikzcd}
	\end{wrapfigure}
	\subsubsection{}\label{subsub:rigidification}
	We now provide some stack anabelian properties. For $\XX$ and $\YY$ Deligne-Mumford stacks, recall that we denote by $\Isom_{G_K} (\Pi_\YY,\Pi_\XX)/\sim_{\Delta_\XX}$ the $\pi_1^{\rm et}(\XX\times\bar{K})$-conjugacy orbit of the isomorphisms $\Pi_\YY\to\Pi_\XX$ that commute over $G_K$ -- and similarly the automorphism version.
	
	\medskip
	
	From the immediate commutativity of the diagram on the right, 
	where $\Delta_{\XXrig}\simeq \Delta_{\XX}/\II_{\XX,\eta}$ and $\Delta_{\XXcoa}\simeq \Delta_\XX/N_\XX$, we then obtain
	\begin{equation}\label{diag:AnabIsomPiX}
		\begin{tikzcd}
			\Isom_K(\YY,\XX) \ar[r,"\Psi"]\ar[d]& \Isom_{G_K}(\Pi_{\YY},\Pi_{\XX})/\sim_{\Delta_{\XX}}\ar[d,two heads]\\
			\Isom_K(\YYrig,\XXrig) \ar[r,"\Psi_{\rm rig}"]\ar[d]& \Isom_{G_K}(\Pi_{\YYrig},\Pi_{\XXrig})/\sim_{\Delta_{\XXrig}} \ar[d,two heads]\\
			\Isom_K(\YYcoa,\XXcoa) \ar[r,"\Psi_{\rm coarse}"]&\Isom_{G_K}(\Pi_{\YYcoa},\Pi_{\XXcoa})/\sim_{\Delta_{\XXcoa}}
		\end{tikzcd}
	\end{equation}
	whose aim in this section is to study the bijectivity of the horizontal arrows.
	
	\begin{proposition}\label{prop:IrigcoaRec}
		For $\XX$ a hyperbolic Deligne-Mumford curve over a number field $K$, the following hold
		\begin{enumerate}
			\item\label{it:RelAnI} The stack inertia groups of closed points $\II_{\XX,\bar{x}}$ are characterized as the maximal finite nontrivial closed subgroups of $\Delta_\XX$ and the generic stack inertia group $\II_{\XX,\bar{\eta}}$ is the intersection of all such groups. Furthermore, $\II_{\XX,\bar{\eta}}$ is characterized as the unique maximal finite closed normal subgroup of $\Delta_\XX$.
			\item\label{it:PIZ} The generic stack inertia group is such that $Z(\Delta_\XX)<\II_{\XX,\eta}$. 
		\end{enumerate}
		As a consequence, a $G_K$-isomorphism $\Pi_\mathcal{Y}\to\Pi_\XX$ induces an isomorphism $\II_{\YY,\eta}\to \II_{\XX,\eta}$ between generic inertia groups, and a $G_K$-isomorphism $\Pi_{\YYrig}\to \Pi_{\XXrig}$, and some $K$-isomorphisms between the coarse and rigidified spaces.
	\end{proposition}

	In other words, $\Psi_{\rm rig}$ and $\Psi_{\rm coarse}$ are surjective. Note that the group-theoretic properties of the stack inertia groups are \emph{relative and mono-anabelian} -- that is, they depend on the data of the morphism $\Pi_\XX\to\Gk$.

	\begin{proof}
		For \ref{it:RelAnI}, let $I<\Delta_\XX$ be such a group and consider the exact sequence \eqref{Eq:inRig}, with $\rho\colon\Delta_{\XX}\to \Delta_{\XXrig}$. Then by \cite{Moc07} Lemma~2.11, there exists a stack inertia group $\II_{\XXrig,\bar{x}}<\Pi_{\XXrig}$ such that $\II_{\XXrig,\bar{x}}>\rho(I)$, and one further has 
		\[
		1 \longrightarrow \II_{\XX,\eta} \longrightarrow \rho^{-1}(\II_{\XXrig,\bar{x}})\longrightarrow\II_{\XXrig,\overline{x}} \longrightarrow 1
		\]
		with $\rho^{-1}(\II_{\XXrig,\bar{x}})>I$ and equality by the maximality assumption. Finally $I\simeq \II_{\XX,\tilde{x}}$, with $\tilde{x}$ a closed point of $\XX$ such that $\rho(\tilde{x})=\bar{x}$.
		
		Consider $I$ a closed normal subgroup of $\Delta_\XX$. From \cite{AbsTopI} Proposition 2.3~(i) we have that $\Delta_{\XXrig}$ is elastic, so that the image of $I$ in $\Delta_{\XXrig}$ is trivial. It follows that $\II_{\XX,\eta}$ contains $I$ and is maximal among closed normal finite subgroups of $\Delta_\XX$.
		
		For \ref{it:PIZ}, by the exact sequence \eqref{Eq:inRig}, an element of $Z(\Delta_\XX)$ gives an element of $Z(\Delta_{\XXrig})$ which, since $\XXrig$ is a stacky curve, is trivial by the slimness of stacky curves as in Proposition~\ref{prop:Zorbi}.   
		
		A $G_K$-isomorphism $\Pi_{\YY}\to \Pi_{\XX}$ induces an isomorphism between geometric fundamental groups $\Delta_{\YY}\to \Delta_{\XX}$ which itself gives isomorphisms $\II_{\YY,\eta}\to \II_{\XX,\eta}$  (resp. $N_\YY\to N_\XX$) from the group theoretic characterizations of the first part. Then, by corestriction in the homotopy sequence~\eqref{Eq:inRig} (resp. \eqref{Eq:inCoa}), also an isomorphism $\Pi_{\YYrig}\longrightarrow \Pi_{\XXrig}$ (resp. $\Pi_{\YYcoa}\longrightarrow \Pi_{\XXcoa}$).
		
		Following \cite{Hos22} Theorem~4.5, since the Galois geometric condition is automatically satisfied over number fields, the isomorphism $\Pi_{\YYrig}\longrightarrow \Pi_{\XXrig}$ lifts to a unique isomorphism $\YYrig\longrightarrow\XXrig$, which in turn induces a unique isomorphism $\YYcoa\longrightarrow\XXcoa$.
	\end{proof}
	
	While of more restricted scope, see Remark \ref{rem:hypRigAnabRel}~(i), the following is of independent interest.
	
	\begin{proposition}\label{prop:BianabXrig}
		For $\XX$ a Deligne-Mumford curve over a number field $K$, such that $\XXcoa$ is hyperbolic, the mono-anabelian characterization of the stack inertia groups at closed points of $\XX$ implies the bijectivity of $\Psi_{\rm coarse}$ and the weak Grothendieck conjecture for $\XXrig$.
	\end{proposition}
	
	\begin{proof}
		For a given $G_K$-isomorphism $\Pi_{\YYrig}\longrightarrow\Pi_{\XXrig}$, proceeding as in the conclusion of Proposition \ref{prop:IrigcoaRec}, the group-theoretic characterization of the various stack inertia groups induces an isomorphism $\Pi_{\YYcoa}\longrightarrow\Pi_{\XXcoa}$. Since a number field is in particular a sub-$p$-adic field, the bijection $\Psi_{\rm coarse}$ is then given by \cite{Moc03} Theorem~4.12 for hyperbolic curves.
		
		It further follows the description of \cite{GS17} Theorem~1 of $\XXrig$ as the root stack $\XXrig\simeq \sqrt{D/\XXcoa}$ and the group-theoretic characterization of the inertia groups $\II_{\XXrig,x}$'s at closed points of Proposition~\ref{prop:IrigcoaRec}, that a morphism $\phi_{coarse}\colon \YYcoa \longrightarrow \XXcoa$ can be completed into a commutative diagram
		\[\begin{tikzcd}
			\YYrig \ar[d]\ar[r,dotted]& \XXrig\ar[d]\\
			\YYcoa \ar[r,"\phi_{coarse}"']& \XXcoa
		\end{tikzcd}
		\]
		so as to give a map $\Isom_{G_K}(\Pi_{\YYrig},\Pi_{\XXrig})/\sim_{\Delta_{\XXrig}}\to \Isom_K(\YYrig,\XXrig)$ at the rigidification level.
	\end{proof}
	
	Results of Proposition~\ref{prop:IrigcoaRec} readily apply to the one-dimensional moduli spaces $\MM_{0,[4]}$ and $\MM_{1,1}$, as well as to any of the special loci $\MM_{g,[r]}(\gamma)_{\rm kr}$ of \cite{CM23} with $kr$-data fixed so that the special loci is one-dimensional. We record an additional straightforward one-dimensional result.
	\begin{corollary}
		For $\XX=\MM_{0,[4]}$ or $\MM_{1,1}$ one has $Z(\Delta_\XX)=\II_{\XX,\eta}$. 
	\end{corollary}
	
	\begin{proof}
		By Proposition \ref{prop:IrigcoaRec} (ii), one has $Z(\Delta_\XX)\leq\II_{\XX,\eta}$. On the other hand, recall that over $\CC$ the space $\MM_{g,[r]}$ classifies mapping class orbits of analytic structures of Riemann surfaces of genus $g$ with $r$ marked points. The group $Z(\Delta_{\XX})$ is thus a finite subgroup of $\pi_1^{\rm orb}(\XX(\CC))$, which by Nielsen-Kerckhoff Theorem can be realized as the stack inertia group of a closed point of $\XX$. This provides the equality by the minimality property of the generic stack inertia group.
	\end{proof}

	\begin{remark}\label{rem:hypRigAnabRel}\mbox{}
		\begin{enumerate}
			\item A Deligne-Mumford curve $\XX$ and its rigidification $\XXrig$ (equivalently an orbicurve) are equivalently hyperbolic since $\chi(\XX)=\chi(\XXrig)/\Card(\II_{\XX,\eta})$, while a coarse hyperbolicity condition is weaker, since $\chi(\XX)=(\chi(\XXcoa)-\sum (1-1/n_i))/\Card(\II_{\XX,\eta})$.
			\item It is a classical argument that over number fields, absolute bi-anabelian is equivalent to relative bi-anabelian -- see \cite{Pop94} Theorem~2 and also \cite{AbsTopIII} Remark 1.11.1~(ii) -- that is, the $G_K$-invariance and $\sim$-orbit conditions can be dropped in Propositions~\ref{prop:IrigcoaRec} and~\ref{prop:BianabXrig}.%
		\end{enumerate}
	\end{remark}
	
	Lifting the results of Proposition~\ref{prop:IrigcoaRec} to $\XX$, in terms of a gerbe $\XX\to \XXrig$ or in terms of bi-anabelian result, is the goal of Section~\ref{subsub:stackgerbes} below.
	
	\subsubsection{} \label{subsub:stackgerbes} 
	We start by reviewing the notion of stacks in gerbes, that is a map of Deligne-Mumford stacks $\XX\to \YY$ that is an étale gerbe in the sense of Definition~\ref{def:etalgerbe}. Let $\XX$ be an irreducible Deligne-Mumford stack with generic inertia $\II_{\XX,\eta}$. Our first goal is to show that the étale gerbe structure of $\XX\to \XX_{\mathrm{rig}}$ is entirely determined by the étale homotopy rigidification sequence. 
	
	\begin{definition} \label{def:etalgerbe}
		Let $\underline{G}$ be a sheaf of finite groups over a Deligne-Mumford stack $\YY$. A map of Deligne-Mumford stacks $\XX\to \YY$ is an étale $\underline{G}$-gerbe if for all schemes $T$ and all objects $\xi\in \YY(T)$, the pullback stack $\XX\times_{\YY} T$ is an étale $\xi^*\underline{G}$-gerbe over $T$. 
	\end{definition}
	
	It follows from the definition that for $x\in \XX(T)$ and for some scheme $T$ with image $y\in \YY(T)$, there is a short exact sequence of finite groups 
	\[
	\begin{tikzcd}
		1\arrow[r] & \underline{G}(y) \arrow[r] & \Aut x \arrow[r] & \Aut y \arrow[r] & 1.
	\end{tikzcd}
	\]
	Indeed, the pullback $\XX \times_{\YY} T$ is a neutral $G_{y}$-gerbe: for $h\in \Aut y$ there are objects $z_1=(x, \operatorname{id}\colon T\to T, h)$ and $z_2=(x,\operatorname{id}\colon T\to T, \mathrm{id})$ of $\XX \times_{\YY} T(T)$, so that by assumption the set $\Isom(z_1,z_2)$ is a $G_y$-torsor. Thus, $\XX \times_{\YY} T$ is non-empty and its elements are automorphisms $u\in \Aut x $ such that the square 
	\[
	\begin{tikzcd}
		\rho(x) \arrow[r, "\rho(u)"] \arrow[d, "h"] & \rho(x) \arrow[d, "\operatorname{id}"] \\
		y \arrow[r, "\operatorname{id}"] & y
	\end{tikzcd}
	\]
	commutes. This gives exactly that $\Aut x\to \Aut y$ is surjective with fibers that are $G_y$-torsors.
	
	\medskip
	
	In the special case where $\II_{\XX,\eta}$ is abelian and where $\XXrig$ is a scheme, the proof is substantially simplified. In that case we have a canonical isomorphism $H^2(\XXrig, \II_{\XX,\eta})\simeq H^2(\Pi_{\XXrig},\II_{\XX,\eta})$ and the gerbe $\XX\to \XXrig$ is determined by the class of the extension given by the rigidification exact sequence of Proposition~\ref{prop:HoRigSeq}.
	
	\medskip
	
	The following result can be seen as an algebraic and homotopical refinement of \cite{MR2279100} Propositions~4.6 and 7.7 that presents an \emph{analytic} Deligne-Mumford curve as a schematic global quotient. The intricacy of its formulation reflects the mono-anabelian and group-theoretic reconstruction nature of the result.
	
	\begin{proposition} \label{prop:monoXoverXrig}
		Let $\XX$ be a Deligne-Mumford curve with generic inertia $\II_{\XX,\eta}$. Then there is an isomorphism $\XX\simeq [U/H]$ over $\XXrig$, where $U$ is a scheme and $H$ is a finite group such that:
		\begin{enumerate}
			\item The scheme $U$ corresponds to the Galois closure of a covering that trivializes the class of the gerbe induced by pullback of $\XX\to \XXrig$ to a covering $V\to \XXrig$, where $V$ is a scheme given by Fenchel-Nielsen theorem, with induced trivial outer action $\pi_1(V)\to \operatorname{Out} \II_{\XX,\eta}$;
			\item The group $H$ is the quotient $\Pi_\XX/\pi_1(U)$.
		\end{enumerate} 
		Furthermore $\XX\simeq  [[U/\II_{\XX,\eta}]/G]$ where $G$ is the quotient $\Pi_\XX/(\II_{\XX,\eta}\times \pi_1(U))$.
	\end{proposition}
	
	We also refer to Theorem~\ref{theo:DMbiAnab} for a bi-anabelian statement. 
	
	\medskip
	
	The characterization of $\XX$ as a neutral gerbe over $\XXrig$ essentially follows Giraud's characterization by an element of $H^1(\XXrig, \Out(\II_{\XX,\eta}))\simeq \Hom (\Pi_{\XXrig}, \Out(\II_{\XX,\eta}))$ by \cite{Gir71} Corollaire~1.1.7.3 and Proposition~6.1.2, and by the image of a class $[\alpha]\in H^1(\Pi_{\XXrig},\operatorname{Inn}(\II_{\XX,\eta}))\to H^2(\Pi_{\XXrig},Z(\II_{\XX,\eta}))$ \cite{Gir71} Ch.~IV Prop.~3.2.6, where the class $[\beta]$ is canonically given by the exact sequence of Proposition~\ref{prop:HoRigSeq}.
	
	\begin{proof}
		We build the scheme $U$ starting from the rigidification exact sequence 
		\begin{equation}\label{Eq:inRig2}
			1 \longrightarrow \II_{\XX,\eta} \longrightarrow \pi_1^{\rm et}(\XX) \longrightarrow \pi_1^{\rm et}(\XXrig) \longrightarrow 1
		\end{equation}
		in a series of steps.
		
		\medskip
		
		Let $U_1\to \XXrig$ be the cover associated with the kernel of the outer action $\Pi_{\XXrig}\to \Out \II_{\XX, \eta}$. We thus have an exact sequence
		\begin{equation}
			1 \longrightarrow \II_{\XX,\eta} \longrightarrow \pi_1^{\rm et}(U_1\times_{\XXrig}\XX) \longrightarrow \pi_1^{\rm et}(U_1) \longrightarrow 1
		\end{equation}
		with trivial induced outer action $\pi_1(U_1\times_{\XXrig} \XX) \to \Out \II_{\XX,\eta}$. Denoting by $E(\pi_1(U_1), \II_{\XX,\eta}, 1)$ the set of extensions of $\pi_1(U_1)$ by $\II_{\XX,\eta}$ with trivial outer action, one obtains a canonical bijection $E(\pi_1(U_1), \II_{\XX,\eta}, 1)\simeq H^2(\pi_1(U_1),Z(\II_{\XX,\eta}))$ that is given by the choice of the direct product as the class of $0$. 
		
		\medskip
		
		Let $U_2\to \XXrig$ be a schematic cover given by Theorem~\ref{thm:proFFenchels} applied to $\Delta_{\XXrig}$ and let $U_3$ be the finite étale covering scheme  $U_3\to \XXrig$ that corresponds to the intersection $\pi_1(U_1)\cap \pi_1(U_2)$. By construction, the induced outer action $\pi_1(U_3)\to \Out \II_{\XX,\eta}$ is trivial so that our previous choice of the direct product carries on to give a bijection $E(\pi_1(U_3), \II_{\XX,\eta}, 1) \simeq H^2(\pi_1(U_3),Z(\II_{\XX,\eta}))$ as before. 
		
		Now, by the results of \cite{Gir71} recalled before the proof, the group $H^2(\pi_1(U_3),Z(\II_{\XX,\eta})) \simeq H^2_{\rm et}(U_3, Z(\II_{\XX,\eta}))$ classifies the $\II_{\XX,\eta}$-gerbes over the scheme $U_3$. Our choice of bijection sends the trivial gerbe to the direct product, which is compatible with the étale fundamental group functor, so that the gerbe $U_3\times_{\XXrig} \XX \to U_3$ is given by the class $\alpha\in H^2(\pi_1(U_3),Z(\II_{\XX,\eta}))$ of the extension
		\begin{equation}
			1 \longrightarrow \II_{\XX,\eta} \longrightarrow \pi_1^{\rm et}(U_3\times_{\XXrig}\XX) \longrightarrow \pi_1^{\rm et}(U_3) \longrightarrow 1.
		\end{equation}
		
		\medskip
		
		As we are dealing with the cohomology of profinite groups, there is a cover $U_4\to U_3\to  \XXrig$ such that the image of $\alpha$ by the natural map $H^2(\pi_1(U_3),Z(\II_{\XX,\eta})) \to H^2(\pi_1(U_4),Z(\II_{\XX,\eta}))$ is $0$. In particular, $U_4\times_{\XXrig} \XX\to U_4$ is the trivial gerbe, that is, we have $U_4\times_{\XXrig} \XX\simeq [U_4/\II_{\XX,\eta}]$. 
		
		Let $U_5$ be the Galois closure of $U_4$ over $\XXrig$ so that we have both $\XXrig\simeq [U_5/G]$ for the finite group $G=\Pi_{\XXrig}/\pi_1(U_5)$ and $U_5\times_{\XXrig} \XX\simeq [U_5/\II_{\XX,\eta}]$. As $\pi_1(U_5)< \ker (\Pi_{\XXrig}\to \Out \II_{\XX,\eta})$, there is a map $G\to \Out \II_{\XX,\eta}$ that gives an action of $G$ on $[U_5/\II_{\XX,\eta}]$ which is compatible with its action on $U_5$. The quotient $\big[[U_5/\II_{\XX,\eta}]/G]$ is naturally isomorphic to $\XX$ over $\XXrig$ as we have a cartesian square
		\[
		\begin{tikzcd}
			\left[U_5/\II_{\XX,\eta}\right]  \arrow[d] \arrow[r] & \XX \arrow[d] \\
			U_5 \arrow[r] & \XXrig.
		\end{tikzcd}
		\]
		
		Finally, as we have a map of extensions, we get a commutative diagram 
		\[
		\begin{tikzcd}
			1\arrow[r] & \II_{\XX,\eta} \arrow[r] \arrow[d, equal] &  \pi_1(\left[U_5/\II_{\XX,\eta}\right]) \arrow[r] \arrow[d] & \pi_1(U_5) \arrow[r] \arrow[d] & 1 \\
			1\arrow[r] & \II_{\XX,\eta} \arrow[r] &  \Pi_\XX \arrow[r] \arrow[d] & \Pi_{\XXrig} \arrow[r] \arrow[d]& 1 \\
			& & G \arrow[r, equal] & G
		\end{tikzcd}
		\]
		Since $\pi_1(\left[U_5/\II_{\XX,\eta}\right])$ is a direct product $\II_{\XX,\eta} \times \pi_1(U_5)$, the image of $\pi_1(U_5)$ in $\Pi_\XX$ is an open normal subgroup with quotient $H$, an extension of $G$ by $\II_{\XX,\eta}$. Furthermore, we have $[U_5/H]\simeq \big[[U_5/\II_{\XX,\eta}]/G\big]\simeq \XX$ as desired.
	\end{proof}
	
	The proof of Proposition~\ref{prop:monoXoverXrig} is ``mono-anabelian'' in the sense that we are able to reconstruct the morphism $\XX\to \XXrig$, up to isomorphism, using only the data coming from the étale fundamental group.
	
	\begin{corollary} \label{cor:monotobi}
		Let $\XX$ and $\YY$ be Deligne-Mumford curves over a number field $K$ and let $f\colon \Pi_\XX\to \Pi_\YY$ be a $G_K$-isomorphism. If the induced map $\overline{f}\colon \Pi_{\XXrig}\to \Pi_{\YYrig}$ comes from a geometric $K$-isomorphism $\overline{f}\colon \XXrig\to \YYrig$, then there is a geometric $K$-isomorphism $f\colon \XX\to \YY$ over~$\overline{f}$.
	\end{corollary}
	
	\begin{proof}
		Let $U_{\XX}$ and $U_{\YY}$ be the schemes given by Proposition~\ref{prop:monoXoverXrig} for $\XX$ and $\YY$ respectively with given actions by $H_{\XX}$ and $H_{\YY}$. From the $K$-isomorphism $\overline{f}\colon \XXrig\to \YYrig$ and the $G_K$-isomorphism $\Pi_\XX\simeq \Pi_\YY$ we get a $K$-isomorphism $u\colon U_{\XX}\to U_{\YY}$. The conclusion of Proposition~\ref{prop:IrigcoaRec} ensures further that $\XX$ and $\YY$ have isomorphic generic inertia groups, and it follows that the groups $H_{\XX}$ and $H_{\YY}$ are isomorphic with isomorphic actions on the respective schemes that are compatible with $u$. The quotients are thus isomorphic, and the resulting isomorphisms sit over $\overline{f}$. 
	\end{proof}
	
	\subsubsection{}\label{subsub:DManab} 
	We now put together the content of Section~\ref{subsub:rigidification} with Proposition~\ref{prop:monoXoverXrig} to obtain our main anabelian result for Deligne-Mumford curves. We start with a lemma that deals with the automorphisms of a Deligne-Mumford curve that reduce to the identity on the rigidification. 
	
	\medskip
	
	We first consider the restriction map $\Aut \Pi_\XX \to \Aut \II_{\XX,\eta}$ and we denote by $\R$ the subquotient of $\Out \II_{\XX,\eta}$ under this image modded out by the image of the restriction of the inner automorphisms of $\Pi_\XX$.

	\begin{lemma} \label{lem:bianabDMautX}
		Let $\XX$ be a Deligne-Mumford curve over an algebraically closed field. Then with the notation above the natural map 
		\[
		\Aut_{\XXrig} \XX \longrightarrow \R
		\]
		is surjective. 
	\end{lemma}
	
	\begin{proof}
		Let $U$ (resp. $G$) be the scheme (resp. the group) associated to $\XX$ by Proposition~\ref{prop:monoXoverXrig}, so that $\XX\simeq [[U/\II_{\XX,\eta}]/G]$. Let $f\in \Out \II_{\XX,\eta}$ given as an automorphism $[U/\II_{\XX,\eta}]$ fixing $U$. The map $f$ descends to an automorphism of $\XX$ as a stack over $\XXrig$ if it is equivariant for the stack-action of the finite group $G$. Let us further denote by $\rho$ the strict action. 
		
		By \cite{Rom05} Definition~2.1, the fact that $f$ is $G$-equivariant corresponds to the facts that $f$ commutes with the image of $G$ in $\Out \II_{\XX,\eta}$ and that we have a cocycle condition that we now give explicitly. First, let $s\colon G\to  \operatorname{Inn}\II_{\XX,\eta}$ be a choice of a map such that by choosing a lift $\widetilde{f}$ of $f$ in $\Aut \II_{\XX,\eta}$ we have
		\[
		\widetilde{f} \circ  \rho(g) \circ  \widetilde{f}^{-1}= s_g\circ \rho(g),\text{ where $s$ is a cocycle}.
		\]
		In order to define a set of natural transformations  that verifies the cocycle condition of \cite{Rom05} as in Rem.~2.4 ibid., the cocycle $s$ needs to be lifted to a cocycle with value in $\II_{\XX,\eta}$. From the exact sequence 
		\[
		1\longrightarrow Z(\II_{\XX,\eta}) \longrightarrow \II_{\XX,\eta} \longrightarrow \operatorname{Inn} \II_{\XX,\eta} \longrightarrow 1
		\]
		we get an obstruction class to lifting $s$ in $H^2(G,Z(\II_{\XX,\eta}))$. One can check that this obstruction class corresponds to the obstruction class to lifting $f$ to an automorphism of $\Pi_\XX$ -- see \cite{We71}. 
		
		\medskip
		
		It follows directly that for any element $r\in \R$, there is a lift $f\in \Out(\II_{\XX,\eta})$ that corresponds to an automorphism of $[U/\II_{\XX,\eta}]$. By construction this automorphism descends to $\XX$, and by compatibility with the étale fundamental group functor, we obtain the surjection. 
	\end{proof}
	
	\medskip
	
	Let us consider the image of the natural map 
	\[
	\Aut_{G_K} (\Pi_\XX)/\sim\Delta_X \longrightarrow \R\times \Out_{G_K}(\Pi_{\XXrig})
	\] given by quotient and restriction, that we denote by $\T$. 
	
	\medskip
	
	We can finally establish the weak Grothendieck conjecture for Deligne-Mumford curves over number fields.
	\begin{theorem}\label{theo:DMbiAnab}
		For $\XX$ and $\YY$ Deligne-Mumford curves over a number field $K$ with $\XX$ hyperbolic, we have the following equivalence
		\[
		\Isom_K(\XX,\YY) \neq \varnothing \text{ if and only if } \Isom_{G_K} ( \Pi_\XX, \Pi_\YY) \neq \varnothing.
		\]
		When $\Isom_{G_K} ( \Pi_\XX, \Pi_\YY)/\sim_{\Delta_\YY}$ is furthermore non-empty, we obtain:
		\begin{enumerate}
			\item a bijection $\Isom_{G_K} ( \Pi_\XX, \Pi_\YY)/\sim_{\Delta_\YY} \longrightarrow \Aut_{G_K}(\Pi_\XX)/\sim_{\Delta_\XX}$,
			\item a surjection $\Isom(\XX,\YY) \longrightarrow \T$.
		\end{enumerate}
	\end{theorem}
	
	\begin{proof}
		Let $\varphi\colon \Pi_\XX\to \Pi_\YY$ be an isomorphism. By the conclusion of Proposition~\ref{prop:IrigcoaRec}, the map $\varphi$ induces an isomorphism of exact sequences
		\[
		\begin{tikzcd}
			1 \arrow[r] & \II_{\XX,\eta} \arrow[d] \arrow[r] & \Pi_\XX \arrow[r] \arrow[d, "\varphi"] & \pi_1(\XX_{\mathrm{rig}}) \arrow[r]  \arrow[d]& 1 \\
			1 \arrow[r] & \II_{\YY,\eta} \arrow[r] & \Pi_\YY \arrow[r] & \pi_1(\YY_{\mathrm{rig}}) \arrow[r] & 1
		\end{tikzcd}
		\]
		and there is an isomorphism $f_{\mathrm{rig}}\colon \XXrig\to \YY_{\mathrm{rig}}$ induced by the map on the quotients $\Pi_{\XXrig}\to \pi_1(\YY_{\mathrm{rig}})$. 
		
		\medskip
		
		Thus, Corollary~\ref{cor:monotobi} gives an isomorphism $f:\XX\to \YY$ such that the square
		\[
		\begin{tikzcd}
			\XX \arrow[r, "f",dashed] \arrow[d] & \YY \arrow[d] \\
			\XXrig \arrow[r, "f_{\mathrm{rig}}"'] & \YY_{\mathrm{rig}}
		\end{tikzcd}
		\]
		commutes, which  gives the first part of the statement. For the second part, we apply Lemma~\ref{lem:bianabDMautX}. 
	\end{proof}
	
	\begin{remark}
		In the special case where $\XXrig=\XXcoa$, one further shows, by considering the pullback by the given torsor over $\XXrig$ functor, that the kernel of the map $\Aut \XX\to \T$ is given by $H^1(\XXrig, Z(\II_{\XX,\eta}))=H^1(\Pi_{\XXrig},Z(\II_{\XX,\eta}))$. This coincides with the kernel of Well's map to the product $\R\times \Out \Pi_{\XXrig}$. 
	\end{remark}

	\subsection{The \texorpdfstring{$m$}{5}-step solvable Grothendieck conjecture for stacky curves}\label{subsection:2.2}
	Relying on the various group-theoretic constructions of Sections \ref{sub:Fenchel_ref} and \ref{sec:weaklyperfect-hyperbolic}, which involve an optimal $3$-step characterization of affineness and hyperbolicity, and some $m$-step control of cuspidal growth in the affine case, we establish the main $m$-step anabelian results of this paper.
	
	\begin{wrapfigure}[5]{r}{2.5cm}
		\vspace*{-2em}
		\centering
		$\xymatrix{
			\tilde{X}^{m} \ar[r]^{\tilde{\phi}} \ar[d] & \tilde{X'}^{m} \ar[d] \\
			\mathcal{X} \ar[r]^{\phi} & \mathcal{X}'
		}$
	\end{wrapfigure}
	\subsubsection{} Let us first define the category of geometrically $m$-step Galois coverings of a stack. Let $K$ be a finitely generated field over $\mathbb{Q}$ and fix an algebraic closure $\overline{K}$. For $\XX$ and $\XX'$ stacky curves over $K$, we  write $\tilde{X}^{m} \to \XX$ (resp.~$\tilde{X'}^{m} \to \XX'$) for \emph{the maximal geometrically $m$-step solvable Galois covering of $\XX$} (resp.~$\XX'$) -- Note that $\tilde{X}^{m}$ and $\tilde{X'}^{m}$ are schemes over $\overline{K}$. We further denote by $\Isom_{\overline{K}/K}\!\bigl(\tilde{X}^{m}/\mathcal{X},\tilde{X'}^{m}/\mathcal{X}'\bigr)$ the set of isomorphisms $(\tilde{\phi},\phi)\in \Isom_{\overline{K}}(\tilde{X}^{m},\tilde{X'}^{m})\times\Isom_{K}(\XX,\XX')$ such that the right-hand side diagram is commutative. 
	
	\begin{definition}\label{def:m+nletter2}
		For $m,n\geq 0$ integers, the  image of the  map $\Isom_{G_{K}}( \Pi_{\mathcal{X}}^{\Delta-(m+n)},\Pi_{\mathcal{X}'}^{\Delta-(m+n)})\rightarrow \Isom_{G_{K}}(\Pi_{\mathcal{X}}^{\Delta-m},\Pi_{\mathcal{X}'}^{\Delta-m})$ is denoted
		$\Isom_{\Gk}^{m+n}(\Pi_{\XX}^{\Delta-m},\Pi_{\XX'}^{\Delta-m})$.
	\end{definition}
	
	\begin{lemma}\label{lem:dicrease_isom}
		Let $m\geq 1$ be an integer. For $\mathcal{X}$ and $\mathcal{X}'$ non-perfect stacky curves over a field $K$ that is finitely generated over $\mathbb{Q}$, the image of the natural morphism
		\begin{equation*}
			\Isom_{\overline{K}/K}\!\bigl(\tilde{X}^{m}/\mathcal{X},\tilde{X'}^{m}/\mathcal{X}'\bigr)
			\longrightarrow
			\Isom_{G_{K}}\!\bigl(\Pi_{\mathcal{X}}^{\Delta-m},\Pi_{\mathcal{X}'}^{\Delta-m}\bigr)
		\end{equation*}
		is contained in
		$\Isom_{G_{K}}^{m+n}\!\bigl(\Pi_{\mathcal{X}}^{\Delta-m},\Pi_{\mathcal{X}'}^{\Delta-m}\bigr)$ for any $n\geq 0$.
	\end{lemma}
	
	\begin{proof}
		For any $n\geq 0$, we have the following commutative diagram with natural maps:
		\begin{equation*}
			\xymatrix{
				\Isom_{\overline{K}/K}(\tilde{X}/\mathcal{X},\tilde{X'}/\mathcal{X}')\ar[r]\ar@{->>}[dd]
				&  \Isom_{G_{K}}(\Pi_{\mathcal{X}},\Pi_{\mathcal{X}'})\ar[d]\\
				&\Isom_{G_{K}}(\Pi_{\mathcal{X}}^{\Delta-(m+n)},\Pi_{\mathcal{X}'}^{\Delta-(m+n)})\ar[d]\\
				\Isom_{\overline{K}/K}(\tilde{X}^{m}/\mathcal{X},\tilde{X'}^{m}/\mathcal{X}')\ar[r]
				& \Isom_{G_{K}}(\Pi_{\mathcal{X}}^{\Delta-m},\Pi_{\mathcal{X}'}^{\Delta-m})
			}
		\end{equation*}
		Thus the image of
		$\Isom_{\overline{K}/K}(\tilde{X}^{m}/\mathcal{X},\tilde{X'}^{m}/\mathcal{X}')$
		in
		$\Isom_{G_{K}}(\Pi_{\mathcal{X}}^{\Delta-m},\Pi_{\mathcal{X}'}^{\Delta-m})$
		is contained in the image of
		$\Isom_{G_{K}}(\Pi_{\mathcal{X}}^{\Delta-(m+n)},\Pi_{\mathcal{X}'}^{\Delta-(m+n)})$.
	\end{proof}
	
	The $m$-step anabelian result of Theorem~\ref{thm:mstep_generalGenus} relies on a version of \cite{MR4745885} Lemma~4.11, originally involving a generization process, that we recall for the reader's convenience.
	
	\begin{theorem}\label{th:Yam24mStep}
		Let $U$ and $U'$ be affine hyperbolic curves over a field $K$ finitely generated over $\mathbb{Q}$, and
		assume $r_U\geq 3$ with $(g_U,r_U)\neq (0,3),\, (0,4)$. Then for any $G_K$-isomorphism $\Phi\colon \Pi_{U}^{\Delta-4}\xrightarrow{\ \sim\ } \Pi_{U'}^{\Delta-4}$, there exists a $K$-isomorphism $f^\Phi\colon U\xrightarrow{\ \sim\ } U'$.
	\end{theorem}
	
	The following is a key lemma in our $m$-step anabelian results.
	\begin{lemma}\label{lem:cofH}
		Let $\XX$ be a non-perfect affine hyperbolic stacky curve over a field $K$ finitely generated over $\mathbb{Q}$. Then for $k\geq 2$ (or $k\geq 1$ in the case $(g_\XX,r_\XX)\neq (0,1)$)  the set
		\begin{equation}\label{eq_cofinalset}
			\mathcal{H}_\XX^k=
			\left\{
			H \lesssim \Pi_{\mathcal{X}}^{\Delta-(k+4)} \,\middle|\,
			\begin{array}{l}
				\text{$H$ is a torsion-free open normal subgroup such that} \\
				\text{$\Delta_{\mathcal{X}}^{(k)}/\Delta_{\mathcal{X}}^{(k+4)} \subset H$, $r_{\mathcal{X}_H} \ge 3$, and $(g_{\mathcal{X}_H},r_{\mathcal{X}_H})\notin\{(0,3),(0,4)\}$}
			\end{array}
			\right\}
		\end{equation}
		is non-empty and cofinal in the set of open subgroups $H\triangleleft\Pi_{\mathcal{X}}^{\Delta-(k+4)}$ satisfying $\Delta_{\mathcal{X}}^{(k)}/\Delta_{\mathcal{X}}^{(k+4)} \subset H$, where $\mathcal{X}_H$ denotes the Galois covering of $\mathcal{X}$ corresponding to $H$.
	\end{lemma}
	
	\begin{proof}
		In both cases the cofinality and non-emptiness are a direct consequence of Corollary~\ref{cor:r_greater-1}. In the case $(g_\XX,r_\XX)= (0,1)$, the cofinal set can be initialized with the torsion-free open normal subgroup $H_1< \Pi_{\XX}^{\Delta-(k+4)}$, for $k+4\geq 5$, induced by Lemma~\ref{cor:r_greater}.
	\end{proof}
	
	\subsubsection{} In the following, we say that the Grothendieck $m$-step anabelian conjecture holds for $\XX$ and $\XX'$ if the following equivalence holds
	\begin{equation}\tag*{$(*)_m$}
		\mathcal{X}\cong_{K}\mathcal{X}' \text{ if and only if } \Pi_{\mathcal{X}}^{\Delta-m}\cong_{G_{K}}\Pi_{\mathcal{X}'}^{\Delta-m}.
	\end{equation}
	
	\medskip
	
	We reach the main results of this paper regarding the Grothendieck $m$-step anabelian conjecture for affine hyperbolic stacky curves. 
	\begin{theorem}\label{thm:mstep_generalGenus}
		Let $\mathcal{X}$ and $\mathcal{X}'$ be two stacky curves over a field $K$ finitely generated over $\mathbb{Q}$, so that $\mathcal{X}$ is non-perfect, affine, and hyperbolic, and assume that $(g_\XX,r_\XX)\neq(0,1)$. Then for every $k\geq 1$ the natural map
		\begin{equation*}\label{thm:mstep_generalGenus_map}
			\Isom_{\overline{K}/K}(\tilde{X}^{k}/\mathcal{X},\tilde{X'}^{k}/\mathcal{X}')\to  \Isom^{k+4}_{G_{K}}(\Pi_{\mathcal{X}}^{\Delta-k},\Pi_{\mathcal{X}'}^{\Delta-k})
		\end{equation*}
		is surjective. In particular, the Grothendieck $5$-step conjecture holds, that is:
		\begin{equation}\tag*{$(*)_5$}
			\mathcal{X}\cong_{K}\mathcal{X}' \text{ if and only if }\Pi_{\mathcal{X}}^{\Delta-5}\cong_{G_{K}}\Pi_{\mathcal{X}'}^{\Delta-5}
		\end{equation}
		In the case where $(g_\XX,r_\XX)=(0,1)$, denote $\{n_1,\dots,n_k\}$ the periods of $\XX$, then the above holds if we assume  $\textstyle\prod\nolimits_{i=1}^k n_i/\operatorname{lcm}(n_1,\dots,n_k) \geq 3$. 
		In particular, the $6$-step Grothendieck conjecture holds for any hyperbolic $(g,r)$-data.   
	\end{theorem}
	
	We recall that in the case $(g,r)=(0,1)$, following Lemma~\ref{cor:r_greater}, the number of cusps $r_H$ of any torsion-free open normal subgroup $H\leq \Delta$ is bounded by
	\[
	r_H\leq r_{\DOne}=\prod_{i=1}^k n_i/\operatorname{lcm}(n_1,\dots,n_k).
	\]
	
	\begin{proof}
		Let us first assume $k\geq 2$ and establish the $k+4$-step Grothendieck conjecture, that is, in particular,  the $6$-step Grothendieck conjecture. We consider the following commutative diagram and we construct the dotted diagonal arrow:
		\begin{equation}\label{proof:mstep_generalGenus_diag1}
			\vcenter{
				\xymatrix{		\Isom_{\overline{K}/K}\!\bigl(\tilde{X}^{k+4}/\mathcal{X},\tilde{X'}^{k+4}/\mathcal{X}'\bigr)
					\ar[r]\ar@{->>}[d]
					&		\Isom_{G_{K}}\!\bigl(\Pi_{\mathcal{X}}^{\Delta-(k+4)},\Pi_{\mathcal{X}'}^{\Delta-(k+4)}\bigr)
					\ni \theta
					\ar@{..>}[dl]\ar[d]
					\\
					(\tilde{\phi},\phi)\in
					\Isom_{\overline{K}/K}\!\bigl(\tilde{X}^{k}/\mathcal{X},\tilde{X'}^{k}/\mathcal{X}'\bigr)
					\ar[r]
					&		\Isom_{G_{K}}\!\bigl(\Pi_{\mathcal{X}}^{\Delta-k},\Pi_{\mathcal{X}'}^{\Delta-k}\bigr)
					\ni \theta_{\phi},\,\overline{\theta}
				}
			}
		\end{equation}
		
		We can assume that the group $\Isom_{G_{K}}(\Pi_{\mathcal{X}}^{\Delta-(k+4)},\Pi_{\mathcal{X}'}^{\Delta-(k+4)})$ is non-empty. For a given $\theta\in \Isom_{G_{K}}(\Pi_{\mathcal{X}}^{\Delta-(k+4)},\Pi_{\mathcal{X}'}^{\Delta-(k+4)})$, we first construct the pair $(\tilde{\phi},\phi)$ of isomorphisms.
		By Proposition~\ref{lem:reconstructionsgr}, we have $(g_{\XX},r_{\XX})=(g_{\XX'},r_{\XX'})$, and by Proposition~\ref{prop:3steprecoofaffine}, both $\mathcal{X}$ and $\mathcal{X}'$ are hyperbolic.
		
		By Lemma~\ref{lem:cofH}, one can consider $H<\mathcal{H}_\XX^k$ that is cofinal in our setup, and put $H'\coloneqq \theta(H)\lhd \Pi_{\mathcal{X}'}^{\Delta-(k+4)}$.
		Since $H$ (resp. $H'$) is torsion-free, each $\mathcal{X}_{H}$ (resp. $\mathcal{X}'_{H'}$) is a hyperbolic (non-stacky) curve.
		We have $H^{\Delta-4}\xrightarrow{\sim}\Pi_{\mathcal{X}_{H}}^{\Delta-4}$
		(resp. $H'^{\Delta-4}\xrightarrow{\sim}\Pi_{\mathcal{X}'_{H'}}^{\Delta-4}$) by Lemma~\ref{lemma:mstep}.
		Hence $\theta$ induces an isomorphism $\theta_{H}:\ \Pi_{\mathcal{X}_{H}}^{\Delta-4}\xrightarrow{\ \sim\ }\Pi_{\mathcal{X}'_{H'}}^{\Delta-4}$ over $G_{K}$.
		By the $m$-step solvable Grothendieck conjecture for affine hyperbolic curves of Theorem~\ref{th:Yam24mStep} with $m=k$, there exists a $K_{H}$-isomorphism
		\begin{equation*}
			\phi_{H}:\ \mathcal{X}_{H}
			\xrightarrow{\ \sim\ }
			\mathcal{X}'_{H'},
		\end{equation*}
		where $K_{H}$ stands for the field of definition of $\mathcal{X}_{H}$.
		Moreover, this construction is compatible along further open subgroups. That is, for any element $M\in \mathcal{H}_\XX^k$ such that $M< H$, with $M'=\theta(M)< H'$, the induced isomorphisms $\phi_{M}\colon \mathcal{X}_{M}\xrightarrow{\sim }\mathcal{X}_{M'}$
		fit into the commutative diagram
		\begin{equation*}
			\xymatrix{
				\mathcal{X}_{M}\ar[r]^{\phi_{M}}\ar[d] & \mathcal{X}_{M'}\ar[d]\\
				\mathcal{X}_{H}\ar[r]^{\phi_{H}}&\mathcal{X}'_{H'}
			}
		\end{equation*}
		by \cite{MR4745885} Lemma~4.14.
		Running over all such subgroups $H$, we obtain a $\overline{K}$-isomorphism
		\begin{equation*}
			\tilde{\phi}:\tilde{\mathcal{X}}^{k}
			\xrightarrow{\ \sim\ }
			\tilde{\mathcal{X}'}^{k}.
		\end{equation*}
		Since the quotient stack of $\tilde{\mathcal{X}}^{k}$  (resp. $\tilde{\mathcal{X}'}^{k}$) by $\Pi_{\XX}^{\Delta-k}$ (resp. $\Pi_{\XX'}^{\Delta-k}$) equals $\mathcal{X}$ (resp. $\mathcal{X}'$), we obtain an isomorphism $\phi\colon \mathcal{X}\xrightarrow{\sim}\mathcal{X}'$.
		This finishes the construction of the pair $(\tilde{\phi},\phi)$ of isomorphisms, hence of the dotted arrow of Diagram~\eqref{proof:mstep_generalGenus_diag1}.
		
		\medskip
		
		Finally, we show the commutativity of the lower right triangle in Diagram~\eqref{proof:mstep_generalGenus_diag1}.
		Let $\theta_{\phi}$ and $\overline{\theta}$ be the image of  the pair $(\tilde{\phi},\phi)$ and of $\theta$ in $\Isom_{G_{K}}(\Pi_{\mathcal{X}}^{\Delta-k},\Pi_{\mathcal{X}'}^{\Delta-k})$, respectively.
		Then by the construction of $\phi$, for every such $H$, we have $\theta_{\phi}(H)=\overline{\theta}(H)$.
		Hence, for any open subgroup $G$ of $G_{K}$ and any section $s\colon G\to \Pi_{\mathcal{X}}^{\Delta-k}$ of the projection $\Pi_{\mathcal{X}}^{\Delta-k}\to G_{K}$, we obtain that $\theta_{\phi}(s(G))=\overline{\theta}(s(G))$.
		Since $\theta$ and $\theta_{\phi}$ are isomorphisms over $G_{K}$, $\theta_{\phi}(x)=\overline{\theta}(x)$ holds for any $x\in s(G)$ and
		since $\Pi_{\mathcal{X}}^{\Delta-k}$ is generated by such $s(G)$, that is, we have
		\[
		\Pi_{\XX}^{\Delta-k}=\overline{\langle s(G) \mid G\subset^{op} G_K,~s\colon G\to \Pi_{\XX}^{\Delta-k}\rangle}
		\]
		from Lemma~14.1.1 of Chapter~14.5 of \cite{FJ23}. We therefore obtain that $\theta_{\phi}=\overline{\theta}$ holds.
		
		\medskip
		
		The same construction holds in the two following cases: If we assume furthermore that $(g_\XX,r_\XX)\neq (0,1)$, then the cofinality and non-emptiness of the set $\mathcal{H}_\XX^k$ is still given by Lemma~\ref{lem:cofH}, so that the map
		\[
		\Isom_{\overline{K}/K}(\tilde{X}^{1}/\mathcal{X},\tilde{X'}^{1}/\mathcal{X}')\to  \Isom^{5}_{G_{K}}(\Pi_{\mathcal{X}}^{\Delta-1},\Pi_{\mathcal{X}'}^{\Delta-1})
		\]
		is surjective, and the $5$-step Grothendieck conjecture ensues.
		
		\medskip
		
		In the case $(g_\XX,r_\XX)= (0,1)$, the cofinal set is initialized with the torsion-free open normal subgroup $H_1$ of $\Pi_{\XX}^{\Delta-5}$ given by Lemma~\ref{cor:r_greater}, so that $r_{H_1}\geq 3$ and $(g_{H_1},r_{H_1})\notin \{(0,3),(0,4)\}$. 
	\end{proof}
	
	We also give a group-theoretic corollary of Theorem~\ref{thm:mstep_generalGenus}. For $m\geq 0$, let us write
	\begin{equation*}
		\Aut^{m+n}_{G_{K}}(\Pi^{\Delta-m}_{\mathcal{X}})\coloneq
		\Isom^{m+n}_{G_{K}}(\Pi^{\Delta-m}_{\mathcal{X}},\Pi^{\Delta-m}_{\mathcal{X}}).
	\end{equation*}
	Then the family
	$\{\Aut^{m+n}_{G_{K}}(\Pi^{\Delta-m}_{\mathcal{X}})\}_{n\ge 0}$ forms a decreasing
	sequence of subgroups of $\Aut_{G_{K}}(\Pi^{\Delta-m}_{\mathcal{X}})$, and we show that this sequence stabilizes.
	
	\begin{corollary}
		Let $m\geq 2$ be an integer. For $\mathcal{X}$ a non-perfect affine hyperbolic stacky curve over a field $K$ finitely generated over $\mathbb{Q}$, we have, for any $N \ge m+4$
		\begin{equation*}
			\Aut^{N}_{G_{K}}(\Pi^{\Delta-m}_{\mathcal{X}})
			= \Aut^{m+4}_{G_{K}}(\Pi^{\Delta-m}_{\mathcal{X}}).
		\end{equation*}
	\end{corollary}
	
	\begin{proof}
		The assertion follows immediately from Lemma~\ref{lem:dicrease_isom} and Theorem~\ref{thm:mstep_generalGenus}.
	\end{proof}

	\begin{remark}\label{rem:asbmStep}
		\emph{Absolute $m$-step.} With the same arguments as in Remark~\ref{rem:hypRigAnabRel}~(ii), i.e., since $\Delta_{\mathcal{X}}^m$ is the maximal normal closed finitely generated subgroup of $\Pi_{\mathcal{X}}^{\Delta-m}$, the isomorphism class of $\mathcal{X}$ as $K$-stack can be reconstructed from $\Pi_{\XX}^{\Delta-5}$ or $\Pi_{\XX}^{\Delta-6}$ as in Theorem~\ref{thm:mstep_generalGenus} \emph{without} the structure morphism to $\Gk$. 
	\end{remark}
	
	\subsubsection{}\label{subsub:mGCfails}
	We conclude with some discussion on the $m$-step solvable Grothendieck conjecture for Deligne-Mumford curves with non-trivial generic inertia, that shows that the generic inertia has to be $m$-step solvable for the conjecture to hold -- we further provide an example where, for any $m\geq 3$, the $m+1$-step version holds while the $m$-step version fails.
	
	\medskip
	
	Let us start with a lemma that indicates which property can be read within the fundamental group whenever the coarse and rigidified spaces coincide.
	
	\begin{lemma}\label{lem:finite_rigid_coarse}
		Let $\XX$ be a Deligne--Mumford curve over an algebraically closed field of characteristic $0$. Assume that
		\begin{equation*}
			\Pi_{\XXcoa}\cong \widehat{\FF}_{k}
		\end{equation*}
		for some integer $k\geq 2$.
		Let $m\geq 2$ be an integer.
		Then the following are equivalent:
		\begin{enumerate}
			\item $\XXrig=\XXcoa$.
			\item The image of the closed subgroup $N_{\XX}$, generated by the inertia subgroups of all closed points, in $ \Pi_{\XX}^m$ is finite.
		\end{enumerate}
	\end{lemma}
	
	\begin{proof}
		The implication (i)$\Rightarrow$(ii) is clear, since $N_{\XX}=\II_{\XX,\eta}$ in this case and
		$\II_{\XX,\eta}$ is finite.
		Assume that $\XXrig\neq \XXcoa$ and let us prove the implication (ii)$\Rightarrow$(i).
		By the natural surjections
		\begin{equation*}
			\Pi_{\XX}^m \twoheadrightarrow \Pi_{\XXrig}^m\twoheadrightarrow \Pi_{\XXrig}^2,
		\end{equation*}
		it is sufficient to show that the image of $N_{\XX}$ in $\Pi_{\XXrig}^2$ is infinite.
		Since $ \Pi_{\XXcoa}\cong \widehat{\FF}_{k}$ with $k\geq 2$, the presentation of an affine profinite $F$-group gives an isomorphism
		\begin{equation*}
			\Pi_{\XXrig} \cong \widehat{\FF}_{k}\ast C_{n_{1}}\ast \cdots\ast C_{n_{s}}
		\end{equation*}
		for some integers $s\geq 1$ and $n_{j}\geq 2$.
		Choose a cyclic factor $C_{n_{1}}=\langle \delta\rangle$ and a free generator $x\in \widehat{\FF}_{k}$.
		Then, by the universal property of profinite free products, there exists a continuous surjection
		\begin{equation*}
			\Pi_{\XXrig} \twoheadrightarrow C_{n_{1}} \wr \widehat{\mathbf{Z}}=C_{n_{1}}^{\widehat{\mathbf{Z}}}\rtimes \widehat{\mathbf{Z}},
		\end{equation*}
		sending $\delta$ to the element $(e,0_{\widehat{\mathbf{Z}}})$ (where we define the element $e$ as $e(0)=\delta$ and $e(a)=1$ for $a\neq 0$), $x$ to $(\mathrm{id}_{C_{n_{1}}^{\widehat{\mathbf{Z}}}}, 1_{\widehat{\mathbf{Z}}})$, and all remaining generators to $1$ (see \cite[Appendix D]{MR2599132}).
		Since the wreath product $C_{n_{1}} \wr \widehat{\mathbf{Z}}$ is metabelian, this morphism factors through $\Pi_{\XXrig}^m$.
		Since $N_{\XX}$ is the closed normal subgroup generated by all torsion elements of $\Pi_{\XXrig}$, its image in $C_{n_{1}} \wr \widehat{\mathbf{Z}}$ contains $C_{n_{1}}^{\widehat{\mathbf{Z}}}$, and hence is infinite.
	\end{proof}
	
	\begin{proposition} \label{prop:monoDMmstep} 
		Let $\XX$ be a Deligne-Mumford curve over an algebraically closed field. Let $G$ be a finite group and $\widehat{\FF}_k$ be the free profinite group of rank $k\geq 2$. Let $m\geq 2$ be an integer, and assume that we have an exact sequence
		\[
		1 \to G \to \Delta_\XX^{m}
		\to \widehat{\FF}_{k}^{m} \to 1.
		\]
		Then, we have an isomorphism $\Delta_{\XXcoa}\simeq \widehat{\FF}_{k}$. Furthermore, we have that $\XXrig=\XXcoa$ and the image of the generic inertia group of $\XX$ in $\Delta^m_{\XX}$ coincides with $G$.
	\end{proposition}
	
	\begin{proof}
		Let $N_{\XX}$ be the closed subgroup of $\Delta_{\XX}$ generated by all closed inertia subgroups and such that $\Delta_\XX/N_{\XX}\simeq \Delta_{\XXcoa}$. By \cite{Yam26} Theorem~2.9, the group $\widehat{\FF}_{k}^{m}$ is torsion-free, thus any torsion element of $\Delta_\XX^{m}$ is in $G$. Since $N_{\XX}$ is the closed normal subgroup generated by all torsion elements of $\Delta_\XX$, its image $\mathrm{Im}(N_{\XX})$ in $\Delta_\XX^{m}$ is generated by torsion elements; thus, we get  $\mathrm{Im}(N_{\XX})< G$. We further obtain a commutative diagram with exact rows
		\[
		\xymatrix{
			1 \ar[r] &
			N_{\XX} \ar[r] \ar[d] &
			\Delta_\XX \ar[r] \ar@{->>}[d] &
			\Delta_{\XXcoa} \ar[r] \ar@{->>}[d] &
			1 \\
			1 \ar[r] &
			G \ar[r] &
			\Delta_\XX^{m} \ar[r] &
			\widehat{\FF}_{k}^{m} \ar[r] &
			1.
		}
		\]
		Since $\widehat{\FF}_{k}^{m}$, $m\geq 2$, is non-abelian, $\Delta_\XX^{m}$ is also non-abelian, thus $\Delta_{\XXcoa}$ is non-abelian, so $\XXcoa$ is hyperbolic. By taking the maximal $m$-step solvable quotient of the upper horizontal exact sequence in the diagram above, we get the commutative diagram with exact rows
		\begin{equation}\label{comdiag3.13.2}
			\xymatrix{
				&
				N_{\XX}^{m} \ar[r] \ar[d] &
				\Delta_\XX^{m} \ar[r] \ar[d]^{\cong} &
				\Delta_{\XXcoa}^{m} \ar[r] \ar@{->>}[d] &
				1 \\
				1 \ar[r] &
				G \ar[r] &
				\Delta_\XX^{m} \ar[r] &
				\widehat{\FF}_{k}^{m} \ar[r] &
				1.
			}
		\end{equation}
		Therefore, by the snake lemma, we obtain that $\ker(\Delta_{\XXcoa}^{m} \twoheadrightarrow \widehat{\FF}_{k}^{m})$ is a finite group. On the other hand, since $\XXcoa$ is hyperbolic, we have that $\Delta_{\XX_{\rm coarse}}^{m}$ is torsion-free by \cite{Yam26} Theorem~2.9. Then  $\Delta_{\XXcoa}^{m}\cong \widehat{\FF}_{k}^{m}$ follows. The calculation of Chen ranks implies that the assignment
		\[
		\{\widehat{\FF}_{r},\ \widehat{S}_{g}\mid r\geq 2, g\geq 2\}
		\longrightarrow
		\{\FinStepSolvQuo{\widehat{\FF}_{r}}{m},\ \FinStepSolvQuo{\widehat{S}_{g}}{m}\mid r\geq 2, g\geq 2\}
		\]
		is injective.
		Hence  $\Delta_{\XXcoa}^{m}\cong \widehat{\FF}_{k}^{m}$ implies  $\Delta_{\XXcoa}\cong \widehat{\FF}_{k}$.
		
		\medskip
		
		By Lemma~\ref{lem:finite_rigid_coarse}, and Diagram~\ref{comdiag3.13.2}, the rest of the conclusion follows.
	\end{proof}
	
	We conclude with the following anabelian result that illustrates the case of affine and global quotient Deligne-Mumford curves with equal rigidification and coarsification.
	
	\begin{proposition} \label{prop:biDMmstep} 
		Let $m\geq 4$ be an integer and $\XX=[X/G]$ affine hyperbolic Deligne-Mumford curve over $K$ where $G$ is a finite $m$-step solvable group acting trivially on $X$. Then any Deligne-Mumford curve $\YY$ such that there is a $G_K$-isomorphism $\Pi_\YY^{\Delta-m}\simeq \Pi_\XX^{\Delta-m}$ gives a commutative square
		\[
		\begin{tikzcd}
			\YY \arrow[r,dashed] \arrow[d] & \XX \arrow[d] \\
			\YYcoa \arrow[r,"\sim"] & \XXcoa
		\end{tikzcd}
		\]
		where the bottom arrow is an isomorphism. 
	\end{proposition}
	
	\begin{proof}
		By Proposition~\ref{prop:monoDMmstep} and the $G_K$-compatibility, we have $\YYrig\simeq \YYcoa$ and $G$ is the maximal $m$-step solvable quotient of the generic inertia group $\II_{\YY,\eta}$. We further get a $G_K$-compatible isomorphism $\Pi_{\YYcoa}^{\Delta-m}\simeq \Pi_X^{\Delta-m}$ so that both coarse spaces are isomorphic. 
		
		\medskip
		
		As $\XX\simeq X\times_{\operatorname{Spec} K} B_{\operatorname{Spec} K} G$ is the trivial $G$-gerbe over $X$, a map $\YY\to \XX$ over $X$ is equivalent to giving a map $\YY\to B_{\operatorname{Spec} K} G$. Such a map corresponds to a $G$-bundle over $\YY$ but by assumption we have a surjective homomorphism
		\[
		\Pi_{\YY} \longrightarrow G
		\]
		which provides such a $G$-bundle over $\YY$. 
	\end{proof}
	
	One can note that, with the notations of the proposition, $\Delta_{\YY}$ is a semi-direct product $\II_{\YY,\eta}\rtimes \widehat{\FF_k}$. We further have that the $m$-step solvable quotient of $\II_{\YY,\eta}$ is $G$ and that the induced action of $\widehat{\FF_k}$ on $\II_{\YY,\eta}^m=G$ is trivial. 
	
	\printbibliography[heading=bibintoc]
	
\end{document}